\documentclass[11pt,a4paper,twoside]{article}

\usepackage{ifthen} 
\newboolean{ShowLabels}
\newboolean{BoxedResults}
	\setboolean{ShowLabels}{false}     

	\setboolean{BoxedResults}{true}    

\usepackage[utf8]{inputenc}
\usepackage[T1]{fontenc}

\usepackage{lipsum} 

\usepackage{amsmath} 	
\usepackage{IEEEtrantools} 
\usepackage{cases} 		
\usepackage{amsthm} 	
\usepackage{amsfonts} 	

\usepackage{calrsfs} 	
\DeclareMathAlphabet{\pazocal}{OMS}{zplm}{m}{n}

\usepackage{dsfont} 	

\usepackage{authblk} 	

\usepackage[dvips]{graphicx}  	
\usepackage{xcolor} 		

\usepackage{comment}

\usepackage{stackengine} 

\usepackage{a4wide} 	

\usepackage{xpatch} 	

\usepackage{lmodern} 	

\usepackage{enumitem} 	

\usepackage{stmaryrd} 	

\usepackage{titletoc} 	

\usepackage{fancyhdr} 	

\usepackage{pdfcomment} 	

\usepackage{soul} 		
\usepackage[linewidth=0.3mm]{mdframed} 

\usepackage{datetime} 
\newdateformat{monthyeardate}{\monthname[\THEMONTH] \THEYEAR} 

\fancypagestyle{TitleStyle}
{
\fancyhf{}

\fancyfoot[L]{
	\rule{\textwidth}{0.4pt}\\ 
	\begin{scriptsize}
	\textsuperscript{1}\,
	Université Paris Cité %
	et Sorbonne Université, %
	\href{https://www.ljll.math.upmc.fr/en/}
	{Laboratoire Jacques-Louis Lions (LJLL)}, %
	Paris, France.\\[1mm]
	\textsuperscript{2}\hspace{\gap}
	Université de Bretagne Occidentale,
	\href{https://www.univ-brest.fr/laboratoire-mathematiques-bretagne-atlantique/en}{Laboratoire de Mathématiques de Bretagne Atlantique (LMBA)},
	Brest,
	France.\\[2mm]
	\end{scriptsize}
	}
}
\fancypagestyle{RegularStyle}{
\fancyhf{} 

\fancyhead[LE,RO]{{\footnotesize \thepage}} 
\fancyhead[CE]{{\footnotesize Emergence of Complexity in Opinion Propagation}}
\fancyhead[CO]{{\footnotesize Romain \textsc{Ducasse} and Samuel \textsc{Tréton}}}
}

\newlength{\myindent}
\dottedcontents{section}[5mm]{}{5mm}{2mm} 

\newcommand{\Croch}[1]{\left[#1\right]}			

\newcommand{\Prth}[1]{\left(#1\right)}			
\newcommand{\prth}[1]{(#1)}				

\newcommand{\glmt}[1]{``#1''}				

\newcommand{\prthh}[1]{\big(#1\big)}			
\newcommand{\prthhhh}[1]{\bigg(#1\bigg)}		

\newcommand{\crochhhh}[1]{\bigg[#1\bigg]}		

\newcommand{\intervalleff}[2]{\left[#1,#2\right]}

\newcommand{\intervalleoo}[2]{\left(#1,#2\right)}

\newcommand{\intervalleE}[2]{\llbracket#1,#2\rrbracket}

\newcommand{\gap}{1mm} 
\newcommand{\gapp}{1mm} 
\newcommand{\gappp}{1mm} 

\newcommand{\N}{\mathbb{N}}

\newcommand{\R}{\mathbb{R}}

\let\sqrtTEMP\sqrt
\def\sqrt#1{\sqrtTEMP{#1\,}}

\newcommand{\point}{\includegraphics[scale=1]{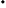}}

\renewcommand{\ast}{\star}

\DeclareRobustCommand{\loongrightarrow}{%
	\DOTSB\relbar\joinrel\relbar\joinrel\rightarrow
}

\newcommand{\itembullet}{\scalebox{0.75}{$\bullet$}}

\newcommand{\indicatrice}[1]{\mathds{1}_{#1}}	

\newcommand{\ie}{i.e.~}

\newcommand\bara[1]{\hspace{0.1mm}{\stackrel{\includegraphics[scale=1]{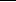}}{#1}}\hspace{0.1mm}}        

\newcommand{\verti}[1]{\left\vert #1 \right\vert}

\newcommand{\Rb}{\pazocal{R}}

\makeatletter
\AtBeginDocument{\xpatchcmd{\@thm}{\thm@headpunct{.}}{\thm@headpunct{}}{}{}}
\makeatother

\newcounter{FIGURE}
\renewcommand{\theFIGURE}{\Roman{FIGURE}}
\newcommand{\eqreff}[2]{(\hyperref[#1]{\ref*{#1}\scalebox{1}{$\,#2$}})}

\newenvironment{system}[5]
{
	\begin{IEEEeqnarray}{#4}
	\IEEEyesnumber #5 \IEEEyessubnumber*
	\smash{\hspace{#1}\left\{\IEEEstrut[#3][#3]\right.}
	\nonumber\\[#2]
}
{
	\end{IEEEeqnarray}
}
\newenvironment{systemB}[6]
{
	\begin{IEEEeqnarray}{#4}
	\IEEEyesnumber #5 \IEEEyessubnumber*
	\smash{\hspace{#1}#6\left\{\IEEEstrut[#3][#3]\right.}
	\nonumber\\[#2]
}
{
	\end{IEEEeqnarray}
}

\definecolor{green}{rgb}{0,.5,0}
\definecolor{brown}{rgb}{.8,0,0}

\newcommand{\OK}[1]{#1}

\ifthenelse{\boolean{ShowLabels}}{
  	\usepackage[notcite,notref]{showkeys} 
  	
}{
}

\newtheoremstyle{mythstyle}
    {}    
    {}    
    {\itshape}  
    {0pt}       
    {\bfseries} 
    {}          
    {5pt plus 1pt minus 1pt} 
    {}          
\theoremstyle{mythstyle} 

	\newtheorem{RAWtheorem}{\textbf{Theorem}}[section]
	\newtheorem{RAWlemma}[RAWtheorem]{\textbf{Lemma}}
	\newtheorem{RAWlemmaNOBREAK}[RAWtheorem]{\textbf{Lemma}}
	\newtheorem{RAWproposition}[RAWtheorem]{\textbf{Proposition}}
	\newtheorem{RAWpropositionNOBREAK}[RAWtheorem]{\textbf{Proposition}}
	\newtheorem{RAWclaim}[RAWtheorem]{\textbf{Claim}}
	\newtheorem{RAWcorollary}[RAWtheorem]{\textbf{Corollary}}
	\newtheorem{RAWdefinition}[RAWtheorem]{\textbf{Definition}}
	\newtheorem{RAWassumption}[RAWtheorem]{\textbf{Assumption}}

\ifthenelse{\boolean{BoxedResults}}{
	\newenvironment{theorem}{\vspace{2mm}\begin{mdframed}\vspace{-2.5mm}\begin{RAWtheorem}}{\end{RAWtheorem}\end{mdframed}\vspace{2mm}}
  	\newenvironment{lemma}{\vspace{3mm}\begin{mdframed}\vspace{-2.5mm}\begin{RAWlemma}}{\end{RAWlemma}\end{mdframed}\vspace{3mm}}
  	\newenvironment{lemmaNOBREAK}{\vspace{3mm}\begin{mdframed}[nobreak=true]\vspace{-2.5mm}\begin{RAWlemmaNOBREAK}}{\end{RAWlemmaNOBREAK}\end{mdframed}\vspace{3mm}}
	\newenvironment{proposition}{\begin{mdframed}\vspace{-3mm}\begin{RAWproposition}}{\end{RAWproposition}\end{mdframed}}
	\newenvironment{propositionNOBREAK}{\begin{mdframed}[nobreak=true]\vspace{-3mm}\begin{RAWpropositionNOBREAK}}{\end{RAWpropositionNOBREAK}\end{mdframed}}
	\newenvironment{claim}{\begin{mdframed}\vspace{-3mm}\begin{RAWclaim}}{\end{RAWclaim}\end{mdframed}}
	\newenvironment{corollary}{\begin{mdframed}\vspace{-3mm}\begin{RAWcorollary}}{\end{RAWcorollary}\end{mdframed}}
	\newenvironment{definition}{\begin{mdframed}\vspace{-3mm}\begin{RAWdefinition}}{\end{RAWdefinition}\end{mdframed}}
	\newenvironment{assumption}{\begin{mdframed}\vspace{-3mm}\begin{RAWassumption}}{\end{RAWassumption}\end{mdframed}}
}{
	\newenvironment{theorem}{\begin{RAWtheorem}}{\end{RAWtheorem}}
	\newenvironment{lemma}{\begin{RAWlemma}}{\end{RAWlemma}}
	\newenvironment{proposition}{\begin{RAWproposition}}{\end{RAWproposition}}

	\newenvironment{definition}{\begin{RAWdefinition}}{\end{RAWdefinition}}
	\newenvironment{assumption}{\begin{RAWassumption}}{\end{RAWassumption}}
}

\newtheorem{remark}[RAWtheorem]{\textbf{Remark}}

\renewenvironment{proof}[1][\proofname]{\hspace{-\myindent}{\bfseries #1.}}{\qed} 

\numberwithin{equation}{section} 

\newcommand{\Sa}[1]{S_{#1}^{\ast}} 	
\newcommand{\Nlast}{N^{\star}} 			
\newcommand{\NN}{N^{\star}}				
\newcommand{\HR}[1]{\hyperref[HR_a]{\pazocal{H}_{#1}}} 	

\newcommand{\MIN}{\hspace{0.3mm} \scalebox{0.8}{$\raisebox{-0.4mm}{$\dfrac{\verti{x}}{\tilde{c}}$}$} \hspace{0.2mm} \scalebox{0.9}{$\wedge$} \hspace{0.4mm} t \hspace{0mm}}

\newcommand{\e}{\varepsilon}

\hypersetup{
	pdftitle={Emergence of complexity in opinion propagation: A reaction-diffusion model},
	pdfauthor={Romain Ducasse and Samuel Tréton},%
	pdfsubject={},%
	pdfkeywords={
			reaction-diffusion systems, %
			SIR models, %
			Fisher-KPP dynamics, %
			spreading speed, %
			propagation properties, %
			social contagion, %
			propagation of opinions}, 
	bookmarksopen=true,
	hidelinks
}

\title{\textbf{Emergence of complexity in opinion propagation:}\\\textbf{A reaction-diffusion model}}

\date{}			

\author{%
	\href{https://sites.google.com/view/romainducasse/}{Romain {\sc Ducasse}}\textsuperscript{1}
	and
	\href{https://www.samueltreton.fr/english/}{Samuel {\sc Tréton}}\textsuperscript{2}
	}

\begin{document}

\maketitle
\thispagestyle{TitleStyle} 



\vspace{-7mm}

\begin{abstract}
We analyze a model designed to describe the spread and accumulation of opinions in a population. Inspired by the {\it social contagion} paradigm, our model is built on the classical SIR model of Kermack and McKendrick and consists in a system of reaction-diffusion equations.
In the scenario we consider, individuals within the population can adopt new opinions {\it via} interactions with others, following some simple rules. The individuals can gradually adopt more complex opinions over time.

Our main result is the characterization of a {\it maximal complexity} of opinions that can persist and propagate. In addition, we show how the parameters of the model influence this maximal complexity. 
Notably, we show that it grows almost exponentially with the size of the population,
suggesting that large communities can foster the emergence of more complex opinions.
\end{abstract}

\vspace{5mm}

\noindent{\textbf{Key words:} %
		reaction-diffusion systems, %
		SIR models, %
		Fisher-KPP dynamics, %
		spreading speed, %
		propagation properties, %
		social contagion, %
		propagation of opinions,
        emergence of complexity.
		}

\vspace{3pt}

\noindent{\textbf{MSC2020:} %
		\pdftooltip{35K40}{PDEs/Parabolic/Second order parabolic systems}, %
		\pdftooltip{35K57}{PDEs/Parabolic/Reaction-diffusion equations}, %
		\pdftooltip{35B40}{PDEs/Qualitative properties/Asymptotic behavior}, %
		\pdftooltip{91D25}{Mathematical sociology/Spatial models in sociology}. %
		}



\section{Introduction}\label{S1_intro}

\subsection{Motivation: the spread of ideas and opinions}

The mathematical modeling of the propagation of ideas, opinions, rumors, knowledge, or other \glmt{social traits} has been envisioned since at least the 18th century \cite{MartinMathematiques02}. The concept of {\it social contagion} is a paradigm used to model such phenomena: building on an analogy between the spread of contagious diseases and of ideas, several authors used modified epidemiological models to study social phenomenon as diverse as the spread of information in a network \cite{RapoportSpread53}, the diffusion of innovations \cite{ BassNew69, BettencourtPower06, ColemanIntroduction64} or the outbreaks of riots \cite{Bonnasse-GahotEpidemiological18, BurbeckDynamics78}.

The key idea behind the analogy in the {\it social contagion} paradigm is that the adoption of a new opinion by an individual occurs \textit{via} interactions with others, much like the transmission of a disease through physical contact. While this analogy has several limitations (e.g., the adoption of a new opinion can be a conscious act, may require repeated voluntary interactions, or involve active efforts), it remains compelling for two reasons. First, it offers mathematicians a range of new models that exhibit behaviors not observed in traditional epidemiological or biological models. Second, these models facilitate quantitative analysis, enabling comparisons with real-world data; see, for instance, \cite{BettencourtPower06, Bonnasse-GahotEpidemiological18}.


Most models in the literature focus on the spread of {\it a single} opinion. However, the capacity of human populations to accumulate opinions and ideas is widely recognized as a key factor in the evolution of ideologies. For sociological insights on this topic, see \cite{LahireStructures23} and references therein.

In this paper, we introduce and analyze a reaction-diffusion model that describes the spread of opinions\footnote{We use the terms {\it opinions} and {\it ideas} interchangeably. A more precise term might be {\it social trait}, referring to a value transmitted between individuals through social interactions. This concept is related to the notion of {\it meme}, introduced by Dawkins in \cite{DawkinsSelfish90}.} within a population, where individuals can gradually adopt increasingly {\it complex} opinions over time. We focus on the following questions: How do different opinions propagate? And is there a limit to the complexity of opinions that can spread?

To formalize our approach, we introduce a general reaction-diffusion model presented below as equation \eqref{syst_gen}, which is built on the following hypotheses:

\begin{enumerate}
	\item The set of all possible opinions is discrete and indexed by $\N$. 


	\item We consider a closed population: no births or deaths occur within the time scales we consider.


	\item The model incorporates a \textit{spatial structure}: the population is distributed across a domain referred to as \glmt{space}. This space may represent a geographic region where the population resides or a social network through which individuals interact. For simplicity, we assume that the space is the entire real line $\mathbb{R}$.


    \item Each individual can exist in one of two states:
	\begin{itemize}[label=\textbullet]
		\item {\it Quiet state}: the individual remains static and is receptive to adopting new opinions from others.
		\item {\it Excited state}: the individual is moving and capable of transmitting its own opinion to others.
	\end{itemize}
	We denote by $S_{n}(t,x)$ and $I_{n}(t,x)$ the densities of quiet and excited individuals, respectively, holding opinion $n \in \mathbb{N}$ at time $t > 0$, and located at point $x \in \mathbb{R}$.\footnote{The choice of the letters $S$ and $I$ is explained later in \hyperref[SS_{1}_2_rap_sir]{Section \ref*{SS_{1}_2_rap_sir}}, and comes from the analogy with epidemiological (Susceptible-Infected-Recovered) models.} 
	
    
	\item When a quiet individual with opinion $n$ encounters an excited one with opinion $k$, the quiet individual may change its opinion from $n$ to $k$.
	We model this process using a law of mass action, assuming that the rate of this opinion adoption at time $t$ and location $x$ is given by $\alpha(k,n)I_k(t,x) S_{n}(t,x)$, where $\alpha(k,n) \geq 0$ represents the likelihood of opinion transmission.
	If $\alpha(k,n) = 0$, it indicates that opinion $k$ cannot be adopted by individuals with opinion $n$.

    \item Once a quiet individual adopts a new opinion, it transitions to the excited state. He can then move and transmit its new opinion to the quiet individuals he would meet.
	This excited state persists for a certain duration before the individual returns to the quiet state.
	We denote $\mu_{n} > 0$ the rate at which an individual holding opinion $n$ ceases to be excited and reverts to the quiet state.
	Consequently, the average duration that an excited individual with opinion $n$ remains active in promoting its newly adopted idea is given by ${1}/{\mu_{n}}$.

	\refstepcounter{FIGURE}\label{FIG_base_model}
	\begin{center}
	\includegraphics[scale=1]{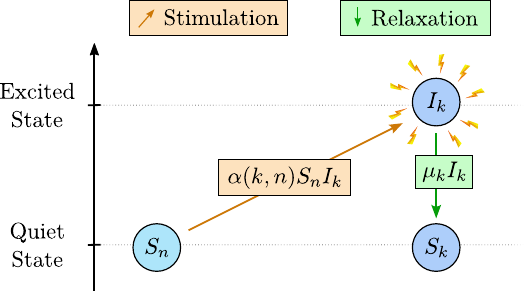}\\[2mm]
	\begin{minipage}[c]{145mm}
	\begin{center}
	\begin{footnotesize}
	\textsl{\textbf{Figure \theFIGURE} ---
		Illustration of the opinion transmission process described in the items 4.-5.-6. above.  
		}
	\end{footnotesize}
	\end{center}
	\end{minipage}
	\end{center}

    \item When individuals are in the excited state, they move randomly in space, following \textit{independent Brownian motions}.
	This movement is modeled using Laplace operators at the macroscopic scale, driving the spatial propagation of opinions.%
	\footnote{Alternative approaches incorporate nonlocal interactions instead of individual diffusion \cite{DiekmannThresholds78, DucThresh}. We reserve the exploration of such insights for future works.}.
\end{enumerate}

Finally, we define the {\it opinion graph} $\mathcal{G}$ as the directed graph whose nodes are the elements of $\N$ (representing the opinions) and whose edges are $\{(n,k) \in \N^2 \, : \, \alpha(k,n)>0 \}$.
This means that an ordered pair of nodes $(n,k)$ is connected if and only if the opinion $k$ can be adopted by individuals having the opinion $n$.

\medskip

By expressing this in the form of equations, we arrive at the following system:
\renewcommand{\gap}{5mm}
\renewcommand{\gapp}{0mm}
\begin{equation}\label{syst_gen}
	\left\lbrace \begin{array}{lllll}
		\partial_{t}S_{n} =  -\Prth{\sum\limits_{k\in \mathbb{N}}^{}\alpha\prth{k,n}I_{k}}S_{n} + \mu_{n}I_{n}, & \hspace{\gap} &
		n\in\mathbb{N}, \hspace{\gapp} & t>0, \hspace{\gapp} & x\in\mathbb{R},\\[5mm]
		\partial_{t}I_{n} = d_{n}\Delta I_{n} + \Prth{\sum\limits_{k\in \mathbb{N}}^{}\alpha\prth{n,k}S_{k}}I_{n} - \mu_{n}I_{n}, & \hspace{\gap} &
		n\in\mathbb{N}, \hspace{\gapp} & t>0, \hspace{\gapp} & x\in\mathbb{R}.
	\end{array} \right .
\end{equation}
This model is rather general, and without further assumptions on the graph $\mathcal{G}$, it can exhibit many different behaviors. In the present paper, we focus our study on the phenomenon of {\it accumulation} of opinions rather than on the process of {\it diversification}: to do so, we shall assume some specific shape on the graph $\mathcal{G}$ that we explain just after. For now, let us mention that the model \eqref{syst_gen} was studied in two situations:

\begin{itemize}
    \item When the opinion graph $\mathcal{G}$ is made of two nodes, that is, when there are only two opinions, say $0$ and $1$, and the individuals with opinion $0$ can adopt the opinion $1$. This situation is equivalent to the classical SIR model of Kermack and McKendrick \cite{KermackContributions91b}, we detail this in \hyperref[SS_{1}_2_rap_sir]{Section \ref*{SS_{1}_2_rap_sir}}, as it will greatly help to build the intuition for our case.
	\refstepcounter{FIGURE}\label{gr0}
	\begin{center}
	\includegraphics[scale=1]{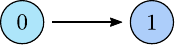}\\[2mm]
	\begin{minipage}[c]{145mm}
	\begin{center}
	\begin{footnotesize}
	\textsl{\textbf{Figure \theFIGURE} ---
		An example of the opinion graph $\mathcal{G}$ for the original SIR model.
		}
	\end{footnotesize}
	\end{center}
	\end{minipage}
	\end{center}
    \item When the opinion graph $\mathcal{G}$ is a tree made of one \glmt{root node} (say $0$), and $N$ \glmt{leaf nodes} connected to $0$ --- see \hyperref[gr1]{Figure \ref*{gr1}} below. This represents the situation where there is one \glmt{neutral} opinion (which cannot be transmitted), while the $N$ other opinions can only be transmitted to individuals holding the neutral opinion. In this configuration, the $N$ different opinions are in competition. The first author showed in \cite{DucassePropagation23} that this model exhibits a {\it selection} phenomenon: only a subset of all the possible opinions spreads.
\refstepcounter{FIGURE}\label{gr1}
\begin{center}
\includegraphics[scale=1]{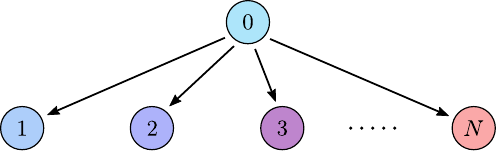}\\[2mm]
\begin{minipage}[c]{145mm}
\begin{center}
\begin{footnotesize}
\textsl{\textbf{Figure \theFIGURE} ---
	An example of the opinion graph $\mathcal{G}$ for the competition of $N$ opinions.
	}
\end{footnotesize}
\end{center}
\end{minipage}
\end{center}
\end{itemize}
In this paper, we study a new instance of \eqref{syst_gen} focused on the accumulation of opinions, by assuming a specific structure on the opinion graph $\mathcal{G}$ --- see \hyperref[gr2]{Figure \ref*{gr2}} below.

\subsection{The model we consider}\label{SS_12_the_model}

Considering a general opinion graph $\mathcal{G}$ in \eqref{syst_gen} leads to rather complex behaviors. A first natural restriction is to consider acyclic graph, that is, graphs that contain no loops.
Mathematically, this means that for any sequence of integer $(n_{1},\cdots,n_{p}) \in \N^p$, we have $\alpha(n_1,n_2)\alpha(n_2,n_3)$ $\cdots\,\alpha(n_{p-1},n_p)\alpha(n_p,n_1)=0$. This is verified in the two situations mentioned above.

We focus here on the situation where the opinion graph $\mathcal{G}$ is $\N$ where the oriented edges are the couples of adjacent points $(n,n+1)$ --- see \hyperref[gr2]{Figure \ref*{gr2}} below.
In this setting, an individual can only adopt opinion $n$ if it holds the previous opinion $n-1$.
Consequently, the integer $n$ serves as a measure of the \textit{opinion's complexity}: the higher $n$, the harder it becomes for an individual to acquire opinion $n$, as it must have acquired all prerequisite opinions $1$, $2$, $\cdots$, $n-1$ beforehand. We see the first opinion $n=0$ acts as a \glmt{basis} opinion, and assume that all individuals possess this opinion initially.

\refstepcounter{FIGURE}\label{gr2}
\begin{center}
\includegraphics[scale=1]{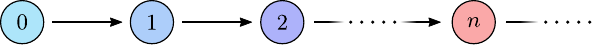}\\[2mm]
\begin{minipage}[c]{145mm}
\begin{center}
\begin{footnotesize}
\textsl{\textbf{Figure \theFIGURE} ---
	The opinion graph $\mathcal{G}$ for the accumulation of opinions.
	}
\end{footnotesize}
\end{center}
\end{minipage}
\end{center}

This setting boils down to assuming, in the general system \eqref{syst_gen}, that $\alpha(k,n) = \alpha_{k}\delta_{k=n+1}$, where $\alpha_k>0$: only excited individuals with opinion $n+1$ may transmit their opinion with individuals with opinion $n$.

\medskip

We are thus led to the following system, which constitutes the main focus of this paper:
\renewcommand{\gap}{17mm}
\renewcommand{\gapp}{4mm}
\vspace{1mm}
\begin{system}{-4.8mm}{-11.6mm}{8.3mm}{lllll}{\label{EQ_MODEL}}
	\partial_{t}S_{0} = -\alpha_{1} S_{0} I_{1}, & &&
	t>0, & x\in\mathbb{R},\label{EQ_MODEL_a}\\[1mm]
	\partial_{t}I_{n} = d_{n}\Delta I_{n} + \alpha_{n}S_{n-1}I_{n} - \mu_{n}I_{n}, & \hspace{\gap} &
	n\in\mathbb{N}^{\star}, \hspace{\gapp} & t>0, \hspace{\gapp} & x\in\mathbb{R},\label{EQ_MODEL_b}\\[-0.5mm]
	\partial_{t}S_{n} = - \alpha_{n+1}S_{n}I_{n+1} + \mu_{n}I_{n}, &&
	n\in\mathbb{N}^{\star}, & t>0, & x\in\mathbb{R}.\label{EQ_MODEL_c}
\end{system}
The basis opinion $n=0$ (the simplest one) cannot be transmitted to other individuals, so there is no need to consider an excited state for $n=0$. For this reason, we do not consider any function $I_0(t,x)$.

This model is specifically designed to focus on the mechanisms of {\it accumulation}, or {\it complexification} of opinions. We believe that, combining the results of the present paper with the results of \cite{DucassePropagation23}, one would be able to treat the case where the opinion graph $\mathcal{G}$ is a general acyclic oriented graph.

\medskip

The dynamics of the model are depicted in \hyperref[FIG_model]{Figure \ref*{FIG_model}} below.
Initially, all the quiet individuals have opinion $n=0$ and reside in the $S_{0}$ compartment.
When a quiet individual with opinion $n=0$ encounters an excited individual with opinion $n=1$, the quiet individual can adopt opinion $n=1$ and transitions from compartment $S_{0}$ to $I_{1}$.
This transition occurs at a rate $\alpha_{1} S_{0} I_{1}$.
Once in the $I_{1}$ compartment, these newly excited individuals remain active for an average duration of ${1}/{\mu_{1}}$, after which they move to the quiet state in compartment $S_{1}$, now holding opinion $n=1$.

The process then continues as such, when a quiet individual with opinion $n$ meets with an excited having opinion $n+1$, he adopts the opinion $n+1$, \ie it transitions from compartment $S_{n}$ to $I_{n+1}$ with rate $\alpha_{n+1} S_{n} I_{n+1}$, resulting in the gradual accumulation of opinions.

Note that interactions only occur between individuals having opinions separated by one degree of complexity (for instance, quiet individuals holding opinion $n=1$ cannot directly adopt the more complex opinion $n=3$).


\refstepcounter{FIGURE}\label{FIG_model}
\begin{center}
\includegraphics[scale=1]{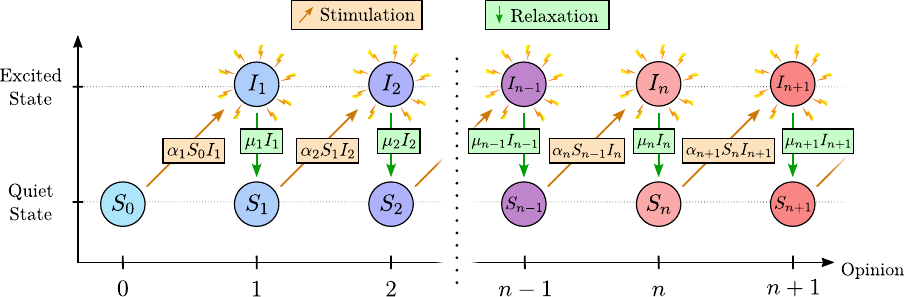}\\[2mm]
\begin{minipage}[c]{145mm}
\begin{center}
\begin{footnotesize}
\textsl{\textbf{Figure \theFIGURE} ---
	Evolution of the population by transitions through the opinion compartments.
	}
\end{footnotesize}
\end{center}
\end{minipage}
\end{center}

One of the main questions is the following:
\begin{center}
    {\bf What is the maximal complexity that the population can achieve?}
\end{center}
More precisely, starting from a population where all the individuals have the opinion $n=0$, will we see the emergence of individuals with opinion $n=1$, $2$, $3$, etc.? If yes, up to which complexity?

Our main result, \hyperref[TH_know_spread_description]{Theorem \ref*{TH_know_spread_description}}, completely describes the long-time behavior of the system, and tells us exactly which opinions spread, and with which {\it speed} (we define this notion below). In particular, depending on the parameters of the system, we give a way to compute the \textit{maximal complexity} that will be reached by the population, that we denote $N^{\star} \in \N\cup\{+\infty\}$.
Although the expression of $\Nlast$ is implicit, we give in \hyperref[th_quali]{Theorem \ref*{th_quali}} some qualitative properties of $\Nlast$, seen as a function of the system's parameters.

\medskip


The dynamical system \eqref{EQ_MODEL} must be supplemented with initial conditions, representing the initial distribution of opinions within the population.
At initial time $t=0$, we assume that only the basic opinion $n=0$ is represented. Specifically, we set
$$
S_{0}(t=0,x) \equiv S_{0}^{\star} \in \mathbb{R}_{+},
\qquad
\text{and}
\qquad
S_{n}(t=0,x)\equiv 0, \qquad \forall n \in \N^{\star}.
$$
Here, $S_{0}^{\star}$ is a nonnegative constant, reflecting a uniform distribution of individuals across space.

To trigger the dynamics, we introduce some {\it small perturbations} by adding a limited quantity of excited individuals for each opinion. We set
$$
I_{n}(t=0,x) = I_{n}^{0}(x), \qquad \forall n \in \N^{\star},
$$
where the functions $I_{n}^{0} : \mathbb{R} \to \mathbb{R}$ are continuous, non-negative, non-zero and compactly supported.
We then investigate which of these perturbations can persist and propagate.

We summarize this initial condition as
\begin{equation}\label{EQ_MODEL_data}
\left\lbrace
\begin{array}{llll}
	S_{0}|_{t=0} \equiv S_{0}^{\star} \in \mathbb{R}_{+}, & \qquad &&
	x\in\mathbb{R},\\[2mm]

 	S_{n}|_{t=0} \equiv 0, & \qquad &
	n\in\mathbb{N}^{\star}, & x\in\mathbb{R},\\[1mm]
	I_{n}|_{t=0} = I_{n}^{0}, & \qquad &
	n\in\mathbb{N}^{\star}, & x\in\mathbb{R}.\\[1mm]
\end{array}
\right .
\end{equation}

\subsection{Background: the SIR reaction-diffusion model}\label{SS_{1}_2_rap_sir}

As we mentioned earlier, both the general system \eqref{syst_gen} and the particular one \eqref{EQ_MODEL} studied in this paper are inspired by mathematical epidemiology, particularly by the SIR (Susceptible-Infected-Recovered) model introduced by Kermack and McKendrick \cite{KermackContributions91b}.
In this section, we present the SIR model together with some related results, as it helps to build intuition and fix the notations for our model.

The SIR model describes the spread of a disease in a population, where individuals are divided into three \textit{compartments} (hence the term compartmental models):
\begin{itemize}[label=\textbullet]
	\item {\it The susceptibles}: they do not have the disease but are at risk of catching it.
	\item {\it The infected}: they are currently infected and can transmit the disease.
	\item {\it The recovered}: they have recovered from the disease and are immune, meaning they cannot be reinfected (no waning of immunity).
\end{itemize}

\medskip

The density of susceptibles, infected and recovered at time $t>0$ and position $x\in\R$ are denoted $S(t,x)$, $I(t,x)$, and $R(t,x)$, respectively. When an infected individual encounters a susceptible, the susceptible can become infected according to a law of mass action, that is at a rate $\alpha S I$, where $\alpha>0$ represents the likelihood of transmission. Infected individuals then recover at a constant rate $\mu>0$.

Although the original SIR model did not account for spatial dynamics, this aspect was later incorporated by several authors. The extension of interest here is the one by Källen \cite{KallenThresholds84}, who assumed that the infected individuals, and only them, diffuse with a diffusivity constant rate $d>0$. This leads to the following system:
\renewcommand{\gap}{10mm}
\renewcommand{\gapp}{2mm}
\begin{equation}\label{model_kallen}
\left\lbrace
\begin{array}{llll}
	\partial_{t}S = -\alpha S I,
	&& t>0, \hspace{\gapp} & x\in\mathbb{R},\\[1mm]
	\partial_{t}I = d  \Delta I + \alpha SI - \mu I,
	&& t>0, \hspace{\gapp} & x\in\mathbb{R},\\[1mm]
	\partial_{t}R =  \mu I,
	&& t>0, \hspace{\gapp} & x\in\mathbb{R}.
\end{array}
\right .
\end{equation}
\begin{remark}\label{req diff} In \cite{KallenThresholds84}, the author chose not to include diffusion term on $S$ in \eqref{model_kallen}, a decision driven by modeling considerations. When there is such a diffusion term, the analysis is more complex, and some points are still open, we refer to \cite{DucrotSpreading21} for more details on this topic.

Here, we follow the same formalism as in the model of Källen: the quiet individual do not diffuse in our model \eqref{EQ_MODEL}, only the excited ones.
\end{remark}


Observe that up to renaming $S$ as $S_{0}$, $I$ as $I_{1}$ and $R$ as $S_{1}$, the SIR system \eqref{model_kallen} is equivalent to our model \eqref{EQ_MODEL}, in the particular case where there are only two opinions.
In other words, the SIR model is the simplest instance of the general model \eqref{EQ_MODEL}.

The SIR system \eqref{model_kallen} is completed with initial datum $(S^{0}(x),I^{0}(x),R^{0}(x))$, which represents the initial spatial distribution of each group of individuals.
We assume $S^{0}(x) \equiv S^{\star}\in \mathbb{R}_{+}$, which means that the susceptibles are initially uniformly distributed across space at the initial time.
Next, we assume $I^{0}$ is non-negative, non-zero, continuous and compactly supported, reflecting the fact that at the initial time, there are only a few infected individuals, and they are localized in space (in the initial \textit{focus of infection}). Finally, we set $R^{0}(x)\equiv 0$; while not necessary, this assumption is natural: at the initial time, no one has recovered yet, as the infection is just beginning.

The main result concerning the SIR system \eqref{model_kallen} is the existence of a {\it threshold effect}.
Let $S_\infty(x) : = \lim_{t\to+\infty}S(t,x)$, representing the {\it final density of susceptibles} that remain unaffected after the epidemic, and define the {\it basic reproduction number} \OK{$\mathcal{R}_0 := \frac{\alpha}{\mu} S^{\star}$}.

\begin{itemize}
    \item If $\mathcal{R}_0 \leq 1$, the disease fades away in the sense that
    $$
    S_\infty(x)\underset{\verti{x} \to +\infty}{\longrightarrow} S^{\star},
    $$
    that is, the final density of susceptible is equal to the initial density, at least for large $\verti{x}$. Hence the disease did not spread.
	Observe that only the large values of $\vert x \vert$ matter: near the initial focus of infection (the support of $I^{0}$), some contamination are inevitable, and $S_\infty$ may even be very low in this region, even if the disease does not spread.

    \item If $\mathcal{R}_0>1$, the disease spreads in the sense that there is $S^{\dagger} \in (0, S_{0})$ such that 
    $$
    \sup_{\vert x \vert >\delta}\vert S_\infty(x) - S^{\dagger}\vert \underset{\delta \to +\infty}{\longrightarrow}0,
    $$
    meaning that the density of susceptible is strictly smaller than the initial density, even far from the initial focus of infection. This indicates that the disease has spread.

	The value $S^\dagger$ is solution of a transcendental equation --- see below for further details.
\end{itemize}

It turns out that we can refine the previous result. When $\mathcal{R}_0>1$, we can characterize the {\it spreading speed} of the epidemic. Define
$$
c^{\star} = 2\sqrt{d(\alpha S^{\star} - \mu) },
$$
then
$$
\sup_{ct < \vert x \vert } \vert S - S^{\star}\vert \underset{t\to+\infty}{\longrightarrow}0, \qquad \text{for all}~ c >c^{\star},
$$
\pagebreak
and
$$
\sup_{\delta < \vert x \vert < ct} \vert S - S^{\dagger}\vert \underset{\delta,t\to+\infty}{\longrightarrow}0, \qquad \text{for all}~ c\in (0,c^{\star}).
$$
The first equation implies that the disease does not spread faster than $c^{\star}$: an observer moving in one direction with speed $c>c^{\star}$ will eventually see around him a density of susceptible equal to the initial value $S^{\star}$, meaning the population there remains unaffected there.

Conversely, the second equation shows that the disease propagates at least as fast as $c^{\star}$: an observer moving at speed $c<c^{\star}$ will observe around him a density of susceptible equal to $S^{\dagger}$ (which is strictly smaller than $S^{\star}$), meaning the population has already been affected before the observer arrives. Thus, the disease spreads faster than the observer.
We say that the function $S(t,x)$ connects $S^{\star}$ to $S^\dagger$ with spreading speed $c^{\star}$.

The propagation of the epidemic is illustrated in \hyperref[FIG_classical_SIR]{Figure \ref*{FIG_classical_SIR}} below, which depicts the situation at some time $t>0$: the population of susceptibles forms two interfaces, while the infected population forms two bumps, both traveling left and right at speed $c^{\star}$. Meanwhile, the recovered population forms a front traveling at the same speed.

\medskip

\refstepcounter{FIGURE}\label{FIG_classical_SIR}
\begin{center}
\includegraphics[scale=1]{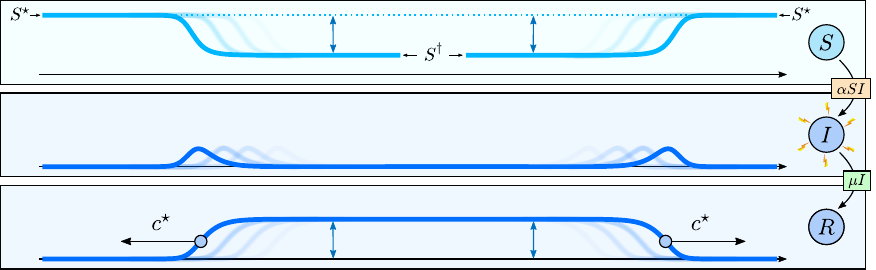}\\[3mm]
\begin{minipage}[c]{145mm}
\begin{center}
\begin{footnotesize}
\textsl{\textbf{Figure \theFIGURE} ---
	The SIR reaction-diffusion model in the case of propagation of the disease ($\mathcal{R}_0>1$).
	}
\end{footnotesize}
\end{center}
\end{minipage}
\end{center}

\medskip

Let us explain how these results are derived, as it will provide insight into the upcoming proofs.
The key idea is to observe that the function $R$ satisfies the equation
\begin{equation}\label{ex kpp}
\partial_{t} R = d \Delta R + \mu S^{\star}(1-e^{-\frac{\alpha}{\mu}R}) - \mu R + \mu I_0,  \qquad t>0, \, x \in \R.
\end{equation}
This comes from the following computations.
First, observe that we can integrate the equation for $S$ in \eqref{model_kallen} to obtain
$$
S(t,x)=S^{\star} e^{-\frac{\alpha}{\mu}R(t,x)}.
$$
Substituting this in the equation for $I$, we get
$$
\partial_{t}I = d  \Delta I + \alpha S^{\star} e^{-\frac{\alpha}{\mu}R(t,x)}I - \mu I.
$$
Now, multiplying by $\mu$ and using the fact that $\partial_{t} R = \mu I$ yields
$$
\partial_{tt}R = d  \partial_{t} \Delta R  + \alpha S^{\star} e^{-\frac{\alpha}{\mu}R(t,x)}\partial_{t} R - \mu \partial_{t} R.
$$
By integrating over $t$, we arrive at equation \eqref{ex kpp}.

\pagebreak

Equation \eqref{ex kpp} is a reaction-diffusion equation whose reaction term
$$f(z)= \mu S^{\star}(1-e^{-\frac{\alpha}{\mu}z})-\mu z$$
is concave.
Such equations are referred to as {\it KPP reaction-diffusion equation}, named after Kolmogorov, Petrovski and Piskunov, who studied them in \cite{KolmogorovStudy37}.
They proved that, if $f^\prime(0)>0$, then the solution of the equation propagates toward the unique positive stationary equilibrium of the equation at speed $2\sqrt{d f^\prime(0)}$. In our case, $f^\prime(0)=\alpha S^{\star} - \mu$, which gives the desired result.

For further details on reaction-diffusion equations, we refer to \cite{AronsonMultidimensional78, KolmogorovStudy37}.
A standard approach to prove these results involves comparing the solution $R$ of \eqref{ex kpp} with subsolutions and supersolutions, which is the main strategy we employ in this paper.

\section{Notations and main results}\label{S2_main_results}

\noindent Before stating our results, in order to justify our notations, we give a heuristic explanation of the dynamics of our system \eqref{EQ_MODEL}.
In \hyperref[FIG_propa_graph]{Figure \ref*{FIG_propa_graph}}, \OK{we represent a typical snapshot of the functions $S_{0}$, $I_{1}$, $S_{1}$, $I_{2}$, etc.,
with each function displayed on a separate graph for clarity. For simplicity, we only depict the right-hand side of the space, as the situation on the left is symmetric.}

At initial time, all individuals are quiet and have opinion $n=0$ \OK{--- \ie they are all in the compartment $S_{0}$.}
\OK{Then, we consider small perturbations. Specifically, for each $n \in \N^{\star}$, a small group of localized excited individuals with opinion $n$ is introduced, with their densities denoted by $I_{n}^{0}$.}

\medskip

Let us start to have a look at the the equations governing the populations $S_{0},I_{1},S_{1}$:
\renewcommand{\gap}{10mm}
\renewcommand{\gapp}{2mm}
\begin{equation}
	\left\lbrace
	\begin{array}{llll}
		\partial_{t} S_{0} = - \alpha_{1} S_{0} I_{1},
		&& t>0, \hspace{\gapp} & x\in\mathbb{R},\\[1mm]
		\partial_{t} I_{1} =  d_{1}\Delta I_{1} + \alpha_{1} S_{0} I_{1} - \mu_{1} I_{1},
		&& t>0, \hspace{\gapp} & x\in\mathbb{R},\\[1mm]
	\partial_{t} S_{1} = \mu_{1} I_{1} - \alpha_2 S_{1} I_{2},
		&& t>0, \hspace{\gapp} & x\in\mathbb{R}.
	\end{array}
	\right .
\end{equation}
This \OK{system} looks like the SIR model,
\OK{where $S_{0}$, $I_{1}$, $S_{1}$ play the role of the susceptible, infected and recovered respectively,}
but with a removal term $-\alpha_2 S_{1} I_{2}$.
If we \OK{disregard} this removal term, \OK{it is natural to expect, similarly to the SIR system}
(see \hyperref[SS_{1}_2_rap_sir]{subsection \ref*{SS_{1}_2_rap_sir}}),
\OK{that when the basic reproduction number $\mathcal{R}_{1} := \frac{\alpha_{1}}{\mu_{1}}S_{0}^{\star}$ exceeds $1$, the density $S_{0}(t,x)$ will decrease from its initial value $S_{0}^{\star}$ towards some value $S_{0}^{\dagger} > 0$.}
\OK{This decrease is expected to occur by the formation of an interface that moves to the right and to the left at a speed $c_1 := 2\sqrt{d_{1}(\alpha_{1} S_{0}^{\star} - \mu_{1})}$, as illustrated in the first line of \hyperref[FIG_propa_graph]{Figure \ref*{FIG_propa_graph}}.}
\OK{Meanwhile, the function $I_{1}$ forms a traveling bump that propagates at the same speed, shown in the second line of the figure.}

\OK{As} $I_{1}$ propagates and becomes quiet, \OK{the population $S_{1}$ emerges, forming a front that connects $0$ to some value $S_{1}^{\star}$, as shown in the third line of \hyperref[FIG_propa_graph]{Figure \ref*{FIG_propa_graph}}.}
\OK{However}, unlike in the SIR system, the $S_{1}$ individuals can \OK{still} be \glmt{contaminated} by the \OK{$I_{2}$ population afterwards}.

\pagebreak

At the initial time $t=0$, there is no population $S_{1}$.
\OK{As a result}, the population $I_{2}$ has no one to \OK{influence}, \OK{so it simply decays and does not propagate}.
However, as the population $S_{1}$ \OK{begins to emerge}, the population $I_{2}$ will find individuals to convince.
More precisely, if we wait \OK{long enough for $S_{1}(t,x)$ to grow towards $S_{1}^{\star}$ over a sufficiently large region}, \OK{then} the population $I_{2}$ \OK{will encounter a local concentration of $S_{1}^{\star}$ individuals (with opinion $n=1$)} to transmit its opinion $n=2$.


\OK{We can then} expect that if the new basic reproduction number
$\mathcal{R}_{2} := \frac{\alpha_2}{\mu_2}S{1}^{\star}$
is strictly greater than $1$, \OK{then} the population $I_{2}$ \OK{will} propagate --- as shown in the fourth line of \hyperref[FIG_propa_graph]{Figure~\ref*{FIG_propa_graph}} --- with some speed
$c_{2} := 2\sqrt{d_2(\alpha_2 S_{0}^{\star} - \mu_2)}$.
\OK{As this happens, the number of individuals holding opinion $n=1$ will decrease: after initially increasing due to the propagation of $I_{1}$, the density $S_{1}(t,x)$ will decline and stabilize at a new value $S_{1}^{\dagger}$, with the interface moving at speed $c_2$.}
This is depicted in the third line of \hyperref[FIG_propa_graph]{Figure \ref*{FIG_propa_graph}}, where $S_{1}$ first increases towards a plateau $S_{1}^{\star}$, approximately in the region $(c_2t, c_1t)$, and then decreases to a new plateau $S_{1}^\dagger$ in the zone $(0, c_2t)$.
We see here that it is natural to assume that $c_2 < c_1$, as otherwise the left-moving interface would catch up with the right-moving one, leading to a degenerate situation.

As the population $I_{2}$ propagates and becomes quiet, the population $S_{2}$ emerges, forming a front that propagates at speed $c_2$, as shown in the fifth line of \hyperref[FIG_propa_graph]{Figure \ref*{FIG_propa_graph}}. \OK{This emergence then gives the} opportunity to the $I_3$ population to propagate, leading to \OK{the appearance of $S_3$, and this process continues in the same manner.}

\OK{If} the basic reproduction number $\mathcal{R}_{n+1} := \frac{\alpha{n+1}}{\mu_{n+1}}S_{n}^{\star}$ is strictly larger than $1$, \OK{this enables the population $I_{n+1}$ to propagate and spread at speed $c_{n+1} := 2\sqrt{d_{n+1}(\alpha_{n+1} S_{n}^{\star} - \mu_{n+1})}$.}
\OK{As a result,} $S_{n}(t,x)$ will decrease as $I_{n+1}$ propagates: for $\vert x \vert < c_{n+1}t$, $S_{n}(t,x)$ \OK{is expected to} converge to some $S_{n+1}^\dagger$, while in the region $\vert x \vert \in (c_{n+1}t, c_n t)$, we should have $S_{n}(t,x) \approx S_{n}^{\star}$. The dynamics of $S_{n}$ \OK{are illustrated in} \hyperref[FIG_propagation_Sn]{Figure \ref*{FIG_propagation_Sn}}.

\refstepcounter{FIGURE}\label{FIG_propagation_Sn}
\begin{center}
\includegraphics[scale=1]{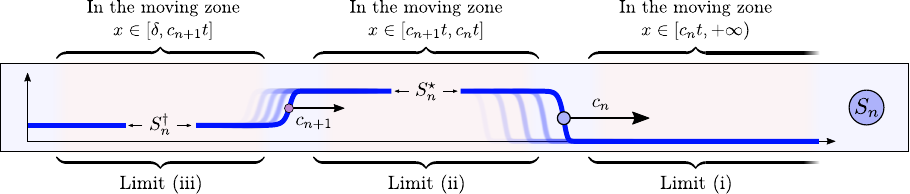}\\[3mm]
\begin{minipage}[c]{145mm}
\begin{center}
\begin{footnotesize}
\textsl{\textbf{Figure \theFIGURE} ---
	Asymptotic shape of $S_{n}$ in the case of propagation $\prth{n\in\intervalleE{1}{\NN}}$. The limits \eqref{EQ_th_i}, \eqref{EQ_th_ii} and \eqref{EQ_th_iii} refer to those of \hyperref[TH_know_spread_description]{Theorem \ref*{TH_know_spread_description}}.
	}
\end{footnotesize}
\end{center}
\end{minipage}
\end{center}
The dynamics may stop at some \OK{value of} $n$. \OK{Suppose there exists a rank, denoted by $N^{\star}$,} such that the basic reproduction number
$\mathcal{R}_{N^{\star}+1} = \frac{\alpha{N^{\star}+1}}{\mu_{N^{\star}+1}} S_{N^{\star}}^{\star}$
is less than or equal to $1$.
Drawing from the intuition developed from the SIR model, it is natural to think that the population $I_{N^{\star}+1}$ will not propagate.

\OK{As a result,} the population $S_{N^{\star}}$ \OK{will} converge \OK{towards} some $S_{N^{\star}}^{\star}$ but, \OK{since it will not be affected by $I_{N^{\star}+1}$, it will not experience a decay at the rear.}
\OK{Instead, it will form a simple front, as depicted in the last line of \hyperref[FIG_propa_graph]{Figure \ref*{FIG_propa_graph}}.}
\OK{Consequently,} all opinions with indices larger than $N^{\star}+1$ \OK{will also fail to propagate.}

\refstepcounter{FIGURE}\label{FIG_propa_graph}
\begin{center}
\includegraphics[scale=0.96]{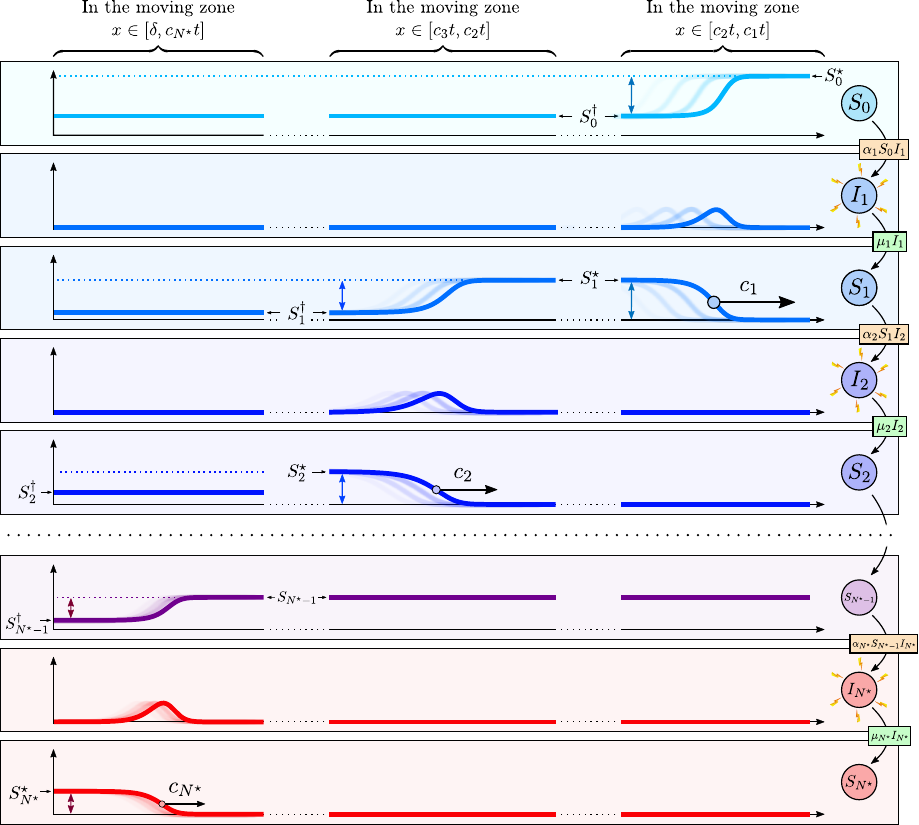}\\[2mm]
\begin{minipage}[c]{145mm}
\begin{center}
\begin{footnotesize}
\textsl{\textbf{Figure \theFIGURE} ---
	\OK{Typical propagation dynamics generated by system \eqref{EQ_MODEL}.}
}
\end{footnotesize}
\end{center}
\end{minipage}
\end{center}

The above heuristics lead us to introduce the following notations.

\begin{definition}[Propagation sequences]
\label{DEF_sequences}
Let $\Sa{0}$ and $\Prth{d_{n}, \alpha_{n}, \mu_{n}}_{n\in\mathbb{N}^{\star}}$ be fixed and strictly positive.\\[-4.5mm]

We define $\Nlast \in \mathbb{N}\cup\{+\infty\}$ and the {\it propagation sequences}
$\Prth{\Sa{n}}_{n\in\intervalleE{0}{\Nlast}}$,
$\Prth{c_{n}}_{n\in\intervalleE{1}{\Nlast}}$, and
$\prth{S_{n}^\dagger}_{n\in\N}$
as follows.\\[-2.5mm]

For each $n\in\mathbb{N}$ such that $\Sa{n}$ is defined ($n=0$ is given by $S_{0}^{\star}$), consider the basic reproduction number
$$
{\Rb}_{n+1} := \frac{\alpha_{n+1}}{\mu_{n+1}} \Sa{n}.
$$
\begin{itemize}
	\item If ${\Rb}_{n+1} \leq 1$, the sequence terminates and we set $\Nlast:=n$.
	\item If ${\Rb}_{n+1} > 1$, we set $\Sa{n+1}$ to be the the unique positive solution to
\begin{equation}\label{EQ_def_S_star}
\Sa{n} \Prth{1 - e^{-\frac{\alpha_{n+1}}{\mu_{n+1}} \Sa{n+1}}} = \Sa{n+1}.
\end{equation}
If the sequence $(S^\star_{n})_{n}$ is not finite, we set $\Nlast = +\infty$.
In addition, we define
$$
c_{n} : = 
2\sqrt{d_{n}\prth{\alpha_{n} \Sa{n-1} - \mu_{n}}}, \qquad \text{for}\ n\in \intervalleE{1}{\Nlast},
$$
and
$$
S_{n}^\dagger : =
\begin{cases}
\Sa{n} - \Sa{n+1} &\text{if }n\in \intervalleE{0}{N^{\star}-1},\\
\Sa{n} &\text{if }n = N^{\star},\\
0 &\text{if }n > N^{\star}.\\
\end{cases}
$$
\end{itemize}
\end{definition}

\noindent Using these notations, we can state our main result, which requires, \OK{as explained above}, the following \OK{assumption.}
\begin{assumption}\label{H1}
The sequence
$\left(c_{n}\right)_{n\in \intervalleE{1}{\Nlast}}$
is strictly decreasing.
\end{assumption}

\begin{remark}
	\hyperref[H1]{Assumption \ref*{H1}} is automatically \OK{satisfied} if all the parameters $(d_n,\alpha_n,\mu_n)_{n\in\N}$ are independent of $n$ \OK{--- \ie $(d_n,\alpha_n,\mu_n)=\prth{d,\alpha,\mu}$ for any $n$}. From a modeling \OK{perspective}, this assumption is quite natural as it implies that more complex opinions propagate slower.
\end{remark}
We are now in position to state our main result. \OK{For the sake of readability}, in the next theorem, we define $c_0 = +\infty$ and $c_{N^{\star}+1} = 0$ (when $\Nlast$ is finite).

\begin{theorem}[Long-term behavior of system \eqref{EQ_MODEL}]\label{TH_know_spread_description}
~

\noindent Let $\prth{S_{0}, \prth{I_{n},S_{n}}_{n\in\mathbb{N}^{\star}}}$ be the solution to
\eqref{EQ_MODEL}-\hspace{0.35mm}\eqref{EQ_MODEL_data}
and assume that
$S_{0}^{\star}$ and\break
$\Prth{d_{n}, \alpha_{n}, \mu_{n}}_{n\in\mathbb{N}^{\star}}$
are such that
\hyperref[H1]{Assumption \ref*{H1}} holds.
Then we have\\[1.5mm]
\renewcommand{\gap}{1mm}
\renewcommand{\gappp}{0.5mm}
\hspace*{\gap}
\itembullet\hspace*{\gappp}
Propagation of the $\NN$ first opinions: for all $\e>0$,
\begin{align}
	\sup_{\vert x \vert>(c_n +\e)t}\vert S_{n}(t,x)\vert\underset{t\to+\infty}{\longrightarrow} 0,
	\qquad \qquad&\forall n\in\intervalleE{1}{\NN},
	\tag{i}\label{EQ_th_i}\\[3mm]
	\sup_{(c_{n+1}+\e)t <\vert x \vert < (c_n-\e)t} \vert S_{n}(t,x) - S_{n}^{\star}\vert \underset{t\to+\infty}{\longrightarrow} 0,
	\qquad \qquad&\forall n\in\intervalleE{0}{\NN},
	\tag{ii}\label{EQ_th_ii}\\[3mm]
	\sup_{\delta<\vert x \vert < (c_{n+1}-\e)t}\vert S_{n}(t,x)-S_{n}^{\dagger}\vert \underset{\delta,t\to+\infty}{\longrightarrow}0,
	\qquad \qquad&\forall n\in\intervalleE{0}{\NN-1}.
	\tag{iii}\label{EQ_th_iii}
\end{align}
\hspace*{\gap}
\itembullet\hspace*{\gappp}
Vanishing of any opinion with complexity strictly higher than $\NN$:
\begin{equation}\label{EQ_th_iv}
	\sup_{\delta<\vert x \vert}\sup_{t>0}\vert S_{n}(t,x)\vert \underset{\delta \to +\infty}{\longrightarrow}0,
	\qquad \qquad \forall n>\NN.
	\tag{iv}
\end{equation}
\end{theorem}
%
\noindent This result \OK{confirms} that the above heuristic \OK{is valid}. Let us state some remarks.
  
\begin{itemize}
	\item $\Nlast \in \mathbb{N} \cup \{+\infty\}$ represents the {\it maximal complexity} of the opinion that can be adopted by the population.

    \item If $N^{\star}=0$, it means that no opinion propagates. In this case, the sequence $\prth{\Sa{n}}_n$ consists of only one element, the sequences $\prth{S_{n}^{\dagger}}_{n}$ is empty, and only the lines \eqref{EQ_th_ii} (with $c_0 = +\infty$ and $c_1 = 0$) and \eqref{EQ_th_iv} need to be considered.

	\item \OK{The propagation of the populations $S_{n}$ varies depending on whether $n=0$, $n \in \llbracket2, \Nlast - 1\rrbracket$, or $n=\Nlast$. When $n=0$, the population $S_{0}$ transitions from a non-zero value $S_{0}^{\star}$ to another value $S_{0}^{\dagger}$. For $n \in \intervalleE{2}{ \Nlast - 1}$, $S_{n}$ connects $0$ to some state $S_{n}^{\star}$, and then to another state $S_{n}^{\dagger}$. Finally, for $n=\Nlast$, the population $S_{\Nlast}$ connects $0$ to some $S_{\Nlast}^{\star}$.}

    \item  We set $c_{\Nlast+1}=0$ \OK{to ensure} that equation \eqref{EQ_th_ii} \OK{remains valid}. \OK{However, this can be refined, as} we will actually prove:
 	$$
		\sup_{\delta <\vert x \vert < (c_n-\e)t} \vert S_{\Nlast}(t,x) - S_{\Nlast}^{\star}\vert \underset{\delta,t\to+\infty}{\longrightarrow} 0.
	$$

    \item A consequence of the theorem is that $\lim_{t\to+\infty} S_n(t,x) \approx S_n^\dagger$, at least for large $\vert x \vert$. This means that $S_n^\dagger$ represents the number of individuals who eventually settle for the opinion $n$.
    
\end{itemize}

\vspace{-2mm}

An important interest of models such as the one considered here is that one can hope to obtain {\it qualitative properties}. In this paper, we are particularly interested in understanding the relationship between the maximal complexity, the size of the population and the strength of the interactions. To do so, we study how the different parameters of the model influence $\Nlast$, the maximal complexity\footnote{For clarity, when needed, we emphasize the dependence of $\Nlast$ with respect to the parameters by writing it as a function of $S_{0}^{\star}$, $(d_n)_{n\in\N^\star}$, $(\alpha_n)_{n\in\N^\star}$ and $(\mu_n)_{n\in\N^\star}$.}.

\vspace{-2mm}

\begin{theorem}[Qualitative properties of the maximal complexity $\NN$]\label{th_quali}
Let $\Nlast \in \N\cup\{+\infty\}$ be as defined in \hyperref[DEF_sequences]{Definition \ref*{DEF_sequences}}.
\begin{itemize}
	\item[$\boxed{1}$] {\bf Monotony of $\Nlast$.} Let us take two set of parameters $S_{0}^{\star}, (d_n)_{n\in\N^\star},(\alpha_n)_{n\in\N^\star},(\mu_n)_{n\in\N^\star}$ and $\bara{S_{0}^{\star}}, (\bara{d_n})_{n\in\N^\star},(\bara{\alpha_n})_{n\in\N^\star},(\bara{\mu_n})_{n\in\N^\star}$ such that the hypothesis \eqref{H1} is verified for each.
	
	\medskip

	If $S_0^\star \leq \bara{S_0^\star}$, $\alpha_n \leq \bara{\alpha_n}$ and $\mu_n \geq \bara{\mu_n}$ for all $n\in \N^\star$, then
	$$
	\Nlast(S_0^\star,\alpha_i,\mu_i) \leq \Nlast(\bara{S_0^\star},\bara{\alpha_i},\bara{\mu_i}).
	$$

    \item[$\boxed{2}$] {\bf Possibility to reach infinite complexity.} There are values of the parameters $(d_n)_{n\in\N^\star}, (\alpha_n)_{n\in\N^\star}, (\mu_n)_{n\in\N^\star}$ for which $\Nlast = +\infty$. Moreover, for any $\varepsilon > 0$, the parameters can be chosen such that 
    $$
    \lim_{n\to +\infty} S_n^\star \geq S_0^\star - \e.
    $$
    \item[$\boxed{3}$] {\bf Asymptotic expression of $\Nlast$.} If the parameters are independent of $n$, meaning there exist constants $d, \alpha, \mu > 0$ such that for all $n \in \N$, we have $d_n = d$, $\alpha_n = \alpha$, and $\mu_n = \mu$, then for large initial populations $S_0^\star$, the following asymptotic equivalent holds:
		\begin{equation}\label{EQ_maximal_complexity}
			\OK{\Nlast(S_0^\star)
			\;
			\underset{{S_{0}^{\star} \to +\infty}}{\scalebox{1.8}{$\displaystyle \sim $}}
			\;
			\frac{e^{\frac{\alpha}{\mu} S_0^\star}}{ \frac{\alpha}{\mu} S_0^\star}.
			}
		\end{equation}
\end{itemize}
\end{theorem}
%
Let us make some remarks on these points. The first one tells us that the quantity $\Nlast$ is a non-decreasing function of the initial population $S_0^\star$ and of the transmission parameters $(\alpha_n)_{n\in\N^\star}$ and a non-increasing function of the recovery parameters $(\mu_n)_{n\in\N}$. This is somewhat natural: the larger the $\alpha_n$, the easier it is for individuals to pass their opinions, the smaller the $\mu_n$, the longer they transmit their opinion. The fact that $\Nlast$ is increasing with respect to the size of the population is also natural: larger populations should indeed establish more connections and would be capable of sharing more complex opinions (the dependence on the size of the population is made more explicit in the third point).

\medskip

The second point indicates that the population can indeed reach opinions with arbitrarily high complexity, and it is even possible for {\it almost all} the initial population to reach arbitrarily large opinions. However, as we shall see in the proof, this requires the coefficients $(\alpha_n)_{n\in\N^\star},(\mu_n)_{n\in\N^\star}$ to be chosen carefully.

\medskip

Although the maximal complexity $\Nlast$ is an implicit function of the parameters, the third point tells us that, when the coefficients do not depend on $n$, $\Nlast$ increases almost exponentially with the size of the initial population. 

This third point can also be expressed using the basic reproduction number $\mathcal{R}_0 = \frac{\alpha S_0^\star}{\mu}$:
\vspace{-2mm}
$$
\Nlast(\mathcal{R}_0)
\;
\underset{{\mathcal{R}_0 \to +\infty}}{\scalebox{1.8}{$\displaystyle \sim $}}
\;
\frac{e^{\mathcal{R}_0}}{\mathcal{R}_0}.
$$
From the modeling point of view, the rapid growth in complexity predicted by the third point seems rather natural. One can indeed expect that large populations should create more interpersonal connections, leading to more interactions, and this should result in higher complexity and diversity of opinions. We refer to \cite{LahireStructures23} for related discussions.

\medskip

To illustrate how the initial density $S_0^\star$ influences the dynamics of the system, let us plot the graph of the function
\vspace{-1mm}
$$
(S_0^\star,n) \in \R_{+} \times \N \mapsto \frac{S_n^\dagger}{S_0^\star}.
$$
This quantity represents the asymptotic proportion of individuals holding opinion $n$, relative to the initial population size $S_{0}^{\star}$, or, to state it differently, is represents the proportion of individuals who eventually adopt the $n$-th opinion\footnote{See the last point below \hyperref[TH_know_spread_description]{Theorem \ref*{TH_know_spread_description}}.}. Because the total number of individuals is conserved, we have $\sum_n \frac{S_n^\dagger}{S_0^\star} =1$. 

If $n\leq \Nlast$, the quantity $\frac{S_n^\dagger}{S_0^\star}$ is strictly positive, and for $n>\Nlast$, this quantity is zero (no individuals adopt these opinions).

In the following graph, the intensity of shaded areas indicates the prevalence of opinion $n$ within the population in long time --- darker shades represent higher proportions.

\refstepcounter{FIGURE}\label{FIG_simu_N_max}
\begin{center}
\includegraphics[scale=0.91]{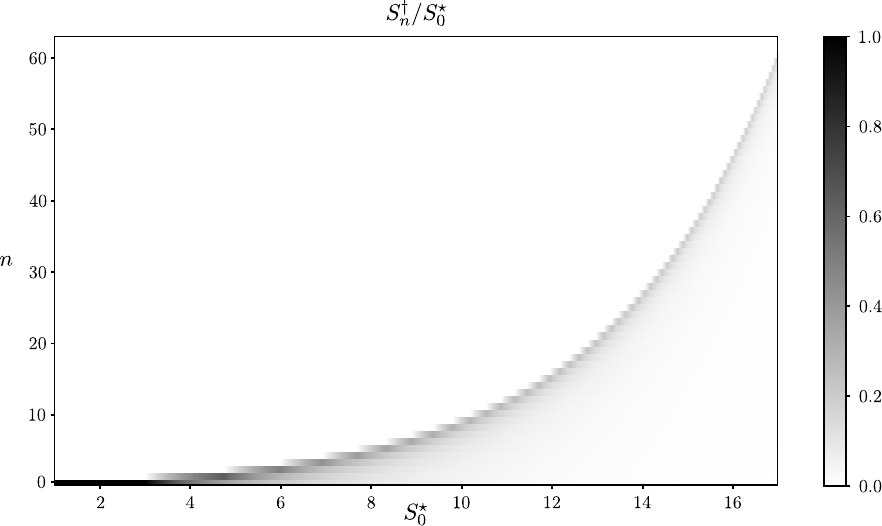}\\[1mm]
\begin{minipage}[c]{145mm}
\begin{center}
\begin{footnotesize}
\textsl{\textbf{Figure \theFIGURE} ---
	Asymptotic proportion ${S_{n}^{\dagger}}/{S_0^{\star}}$ of individuals holding opinion $n$.
}
\end{footnotesize}
\end{center}
\end{minipage}
\end{center}

We observe that the maximal opinion complexity $N^{\star}=N^{\star}\prth{S_{0}^{\star}}$ corresponds to the highest point of non-zero proportion along each column (for $n\leq \Nlast$, the area is gray, while for $n > \Nlast$ the zone is blank). Tracking the curve formed by these maximal points reveals an almost exponential shape, consistent with the predicted behavior of $N^{\star}$ for large $S_{0}^{\star}$, as described in the final point of \hyperref[th_quali]{Theorem \ref*{th_quali}} --- see \eqref{EQ_maximal_complexity}.

\section{\texorpdfstring{Proof of \hyperref[TH_know_spread_description]{Theorem \ref*{TH_know_spread_description}}}{Proof of Theorem \ref*{TH_know_spread_description}}}

The goal of this section is to prove \hyperref[TH_know_spread_description]{Theorem \ref*{TH_know_spread_description}}. As explained in \hyperref[SS_{1}_2_rap_sir]{Section \ref*{SS_{1}_2_rap_sir}}, the classical SIR system can be rewritten as a single scalar reaction-diffusion equation for the recovered. Building on this intuition, we introduce the functions $R_n$, $n\in \N^\star$, defined by
\renewcommand{\gap}{15mm}
\renewcommand{\gapp}{4mm}
\vspace{4mm}
\begin{system}{-4.5mm}{-9.08mm}{5mm}{llll}{\label{EQ_RN}}
	\partial_{t}R_{n} = \mu_{n}I_{n}, & \hspace{\gap}  & t>0, \hspace{\gapp}  &
	x\in\mathbb{R},\label{EQ_RN_a}\\[1mm]
	R_{n}|_{t=0} \equiv 0, & & & x\in\mathbb{R},\label{EQ_RN_b}
\end{system}
~\\[-5mm]
and we will study their spreading properties first instead of studying the functions $S_{n}$ and $I_{n}$. Unlike for the classical SIR system, the functions $R_n$ will not satisfy a scalar reaction-diffusion equations, nor a \glmt{simple} reaction-diffusion system: they will be be coupled, in a somewhat implicit fashion. The core of the proof will be to control the influence of each $R_n$ on the others.

\refstepcounter{FIGURE}\label{FIG_Sn_vs_Rn_schema}
\begin{center}
\includegraphics[scale=1]{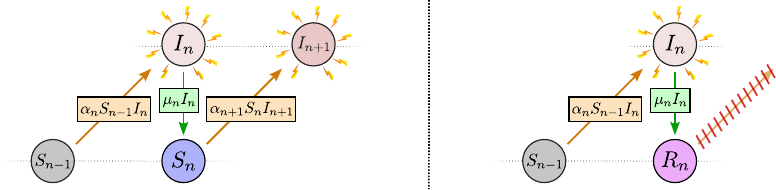}\\[0mm]
\begin{minipage}[c]{145mm}
\begin{footnotesize}
\begin{center}
\textsl{\textbf{Figure \theFIGURE} ---
	$S_{n}$ vs. $R_{n}$. 
}
\end{center}
\end{footnotesize}
\end{minipage}
\end{center}

Our key result concerning the functions $R_n$ is the following.
\begin{proposition}[Propagation of $R_{n}$]
\label{PROP_Rn_propagation}
Assume that the hypotheses of \hyperref[TH_know_spread_description]{Theorem \ref*{TH_know_spread_description}} hold true. Then, for every $n\in \intervalleE{1}{\Nlast}$, we have
\begin{equation}\label{EQ_propo_j}
\sup_{ c t<\vert x\vert} \vert R_n(t,x)\vert \underset{t\to+\infty}{\longrightarrow}0, \qquad \qquad \forall c > c_n, 
\tag{j}
\end{equation}
and
\begin{equation}\label{EQ_propo_jj}
\sup_{\delta<\vert x\vert< c t} \vert R_n(t,x) - S_{n}^{\star}\vert \underset{\delta, t\to+\infty}{\longrightarrow}0, \hspace{7mm}\forall c \in (0,c_n).
\tag{jj}
\end{equation}
\end{proposition}

This proposition tells us that the function $R_n$ spreads toward $S_{n}^{\star}$ with speed $c_n$. The propagation of the function $R_n$ is similar to the propagation we want to prove on the functions $S_{n}$, except it does not have the decay toward $S_{n}^\dagger$ at the back of the front --- see \hyperref[FIG_Sn_vs_Rn]{Figure \ref*{FIG_Sn_vs_Rn}}. For this reason, it will be more convenient to work with the functions $R_n$ rather than $S_{n}$.

\refstepcounter{FIGURE}\label{FIG_Sn_vs_Rn}
\begin{center}
\includegraphics[scale=1]{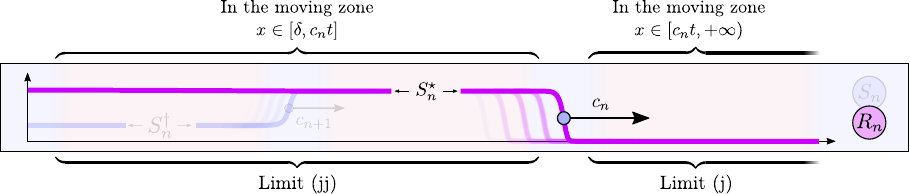}\\[1mm]
\begin{minipage}[c]{145mm}
\begin{footnotesize}
\begin{center}
\textsl{\textbf{Figure \theFIGURE} ---
	 $R_{n}$ behave like $S_{n}$ but without the decay toward $S_{n}^\dagger$ at the back of the front.
}
\end{center}
\end{footnotesize}
\end{minipage}
\end{center}

This section is organized as follows. In \hyperref[sec Rn]{Section \ref*{sec Rn}}, we give some basic estimates on the functions $R_n$. In \hyperref[SS_Rn_perturbed_KPP]{Section \ref*{SS_Rn_perturbed_KPP}} we show that each $R_n$ satisfies a reaction-diffusion equation up to some perturbation term. Then, in \hyperref[SS_Propa_Rn]{Section \ref*{SS_Propa_Rn}} we prove \hyperref[PROP_Rn_propagation]{Proposition \ref*{PROP_Rn_propagation}}. Finally, in \hyperref[SS_proof_of_th]{Section \ref*{SS_proof_of_th}}, we show how \hyperref[PROP_Rn_propagation]{Proposition \ref*{PROP_Rn_propagation}} implies \hyperref[TH_know_spread_description]{Theorem \ref*{TH_know_spread_description}}.

\subsection{Basic results on the auxiliary functions \texorpdfstring{$R_{n}$}{Rn}}\label{sec Rn}

We start this section with a remark concerning the function $R_1$.
\begin{remark}\label{n=1}
    The computations presented in the introduction, at the end of \hyperref[SS_{1}_2_rap_sir]{Section \ref*{SS_{1}_2_rap_sir}}, yield that $R_1$ satisfies the same reaction-diffusion equation than in the case of the classical SIR system (the functions $S_{0}$, $I_{1}$, $R_1$ actually form a SIR system), that is, we have
    $$
    \partial_{t} R_1 = d_{1} \Delta R_1 + f_1(R_1)+\mu_{1} I_{1}^{0}, \qquad t>0, \, x \in \R.
    $$
    Therefore, \hyperref[PROP_Rn_propagation]{Proposition \ref*{PROP_Rn_propagation}} holds true for $n=1$. Our proof of \hyperref[PROP_Rn_propagation]{Proposition \ref*{PROP_Rn_propagation}} will be done by induction and we shall use this as our base case.
\end{remark}

\OK{The next lemma explains the relationship between the functions $R_{n}$ and $S_{n}$.}
\begin{lemmaNOBREAK}[Controlling $S_{n}$ with $R_{n}$]\label{LE_1_control_Sn_with_Rn}
For all $t>0$ and $x\in\mathbb{R}$, we have
\begin{align}
S_{0}^{\star} \, e^{-\frac{\alpha_{1}}{\mu_{1}}R_{1}(t,x)} &=
S_{0}(t,x) \leq S_{0}^{\star}, &
\label{EQ_le1_control_S0} \\[2mm]
R_{n}(t,x) \, e^{-\frac{\alpha_{n+1}}{\mu_{n+1}}R_{n+1}(t,x)} &\leq
S_{n}(t,x) \leq R_{n}(t,x), &
\hspace{-12mm}\forall n \in \mathbb{N}^{\star}. \label{EQ_le1_control_Sn}
\end{align}
\end{lemmaNOBREAK}

\begin{proof}[\texorpdfstring{Proof of \hyperref[LE_1_control_Sn_with_Rn]{Lemma \ref*{LE_1_control_Sn_with_Rn}}}{Proof of Lemma \ref*{LE_1_control_Sn_with_Rn}}]
We start by establishing \eqref{EQ_le1_control_S0}. Because $S_{0}$ and $I_{1}$ are non-negative, it follows from \eqref{EQ_MODEL_a} that
$$
-\alpha_{1}S_{0}I_{1} = \partial_{t}S_{0} \leq 0, \qquad t>0, \, x \in \R.
$$
Dividing this by $S_{0}$ and using the equation for $\partial_{t}R_{1}$ from \eqref{EQ_RN_a}, we obtain
$$
-\frac{\alpha_{1}}{{\mu}_{1}}\partial_{t}R_{1} = \frac{\partial_{t}S_{0}}{S_{0}} \leq 0.
$$
Now we integrate from $0$ to $t$. Recalling that $R_{1}|_{t=0}\equiv 0$, we get
$$
-\frac{\alpha_{1}}{\mu_{1}}R_{1}\prth{t,x} = \ln\prth{S_{0}\prth{t,x}} - \ln\prth{\Sa{0}} \leq 0,
$$
which directly leads to \eqref{EQ_le1_control_S0}.

\medskip

For $n\in \N^{\star}$, to get the upper bound of $S_{n}$,
we consider equation \eqref{EQ_MODEL_c} where $S_{n}$ and $I_{n+1}$ are both positive.
Using $\eqref{EQ_RN_a}$, this leads to
$$
\partial_{t}S_{n} \leq \mu_{n}I_{n} = \partial_{t}R_{n},
$$
Integrating from $0$ to $t$ and using that
$S_{n}|_{t=0}= R_{n}|_{t=0} \equiv 0$,
we get
\begin{equation}\label{EQ_SN_leq_RN}
S_{n}(t,x) \leq R_{n}(t,x)
\end{equation}
which is the upper bound of $S_{n}$ as specified \eqref{EQ_le1_control_Sn}.

To get the lower bound in \eqref{EQ_le1_control_Sn}, we divide \eqref{EQ_MODEL_c} by $S_{n}$ and use \eqref{EQ_RN_a} to get
$$
\frac{\partial_{t}S_{n}}{S_{n}} =
-\frac{\alpha_{n+1}}{\mu_{n+1}}\partial_{t}R_{n+1} +
\frac{\partial_{t}R_{n}}{S_{n}}, \qquad t>0, \, x \in \R.
$$
Using the upper bound \eqref{EQ_SN_leq_RN} on $S_{n}$ and the positivity of $\partial_{t}R_{n} = \mu_{n}I_{n}$, it follows that
\begin{equation}\label{EQ_SN_sur_RN_presque0}
\frac{\partial_{t}S_{n}}{S_{n}} \geq
-\frac{\alpha_{n+1}}{\mu_{n+1}}\partial_{t}R_{n+1} +
\frac{\partial_{t}R_{n}}{R_{n}}.
\end{equation}
We integrate the inequality \eqref{EQ_SN_sur_RN_presque0} from some small $\eta>0$ to $t$. This results in
$$
S_{n}\prth{t,x} \geq
\frac{S_{n}\prth{\eta,x}}{R_{n}\prth{\eta,x}}
\times
R_{n}\prth{t,x} e^{-\frac{\alpha_{n+1}}{\mu_{n+1}}R_{n+1}\prth{t,x}}.
$$
Let us now prove that
$S_{n}\prth{\eta,x}/R_{n}\prth{\eta,x}$
converges to $1$ for all $x$ as $\eta$ goes to zero. By combining \eqref{EQ_MODEL_c} and \eqref{EQ_RN_a}, we obtain
$$
\partial_{t}S_{n} = \partial_{t}R_{n} - \alpha_{n+1} S_{n} I_{n+1}.
$$
Integrating from $0$ to $\eta$ and using again that
$S_{n}|_{t=0}= R_{n}|_{t=0} \equiv 0$,
we are led to
$$
S_{n}\prth{\eta,x} = R_{n}\prth{\eta,x} - \alpha_{n+1} \int_{0}^{\eta} S_{n}\prth{s,x}I_{n+1}\prth{s,x} ds,
$$
which can be rearranged as
\begin{equation}\label{EQ_SN_sur_RN_presque}
\frac{S_{n}\prth{\eta,x}}{R_{n}\prth{\eta,x}} - 1 =
\frac{\alpha_{n+1}}{R_{n}\prth{\eta,x}}\int_{0}^{\eta} S_{n}\prth{s,x}I_{n+1}\prth{s,x} ds.
\end{equation}
Given the positivity and the upper bound on $S_{n}$
established in \eqref{EQ_SN_leq_RN}, as well as the positivity of
$\partial_{t}R_{n} = \mu_{n}I_{n}$,
we have, for any $s\in\intervalleoo{0}{\eta}$,
$$
0 \leq
S_{n}\prth{s,x} \leq
R_{n}\prth{s,x} \leq
R_{n}\prth{\eta,x}.
$$
As a result, \eqref{EQ_SN_sur_RN_presque} implies that
$$
\verti{\frac{S_{n}\prth{\eta,x}}{R_{n}\prth{\eta,x}} - 1} \leq
\alpha_{n+1}\times\eta \times \sup\limits_{s \in \intervalleoo{0}{\eta}}\verti{I_{n+1}\prth{s,x}}
$$
which vanishes as $\eta$ approaches $0$.
\end{proof}

\medskip

The next lemma shows that $R_n$ satisfies some differential inequality involving $R_{n-1}$. In the \hyperref[SS_Rn_perturbed_KPP]{Section \ref*{SS_Rn_perturbed_KPP}}, we will improve this result and prove that $R_n$ \glmt{almost} satisfies a scalar reaction-diffusion equation.

\begin{lemma}[$R_{n}$ is sub-solution to a perturbed Fisher-KPP equation]
	\label{LE_2_Rn_sb_s_KPP_perturbed}
 For all $n\in \N^{\star}$, for all $t>0$ and $x\in\mathbb{R}$, there holds
	\begin{equation}\label{EQ_Rn_sb_sol_KPP_perturbed}
	\partial_{t}R_{n} \leq
	d_{n}\Delta R_{n} +
	\mu_{n} R_{n-1}\!\Prth{1-e^{-\frac{\alpha_{n}}{\mu_{n}}R_{n}}} -
	\mu_{n}R_{n} +
	\mu_{n}I_{n}^{0}.
	\end{equation}
	
\end{lemma}

\begin{proof}[\texorpdfstring{Proof of \hyperref[LE_2_Rn_sb_s_KPP_perturbed]{Lemma \ref*{LE_2_Rn_sb_s_KPP_perturbed}}}{Proof of Lemma \ref*{LE_2_Rn_sb_s_KPP_perturbed}}]
The proof relies on a computation similar to the one presented in \hyperref[SS_{1}_2_rap_sir]{Section \ref*{SS_{1}_2_rap_sir}} to obtain \eqref{ex kpp}: we multiply \eqref{EQ_MODEL_b} by $\mu_{n}$ and using the definition of $R_n$, \eqref{EQ_RN_a}, we find
$$
\partial_{tt}R_{n} = d_{n} \Delta\partial_{t}R_{n} + \mu_{n}\alpha_{n}S_{n-1}I_{n} - \mu_{n}\partial_{t}R_{n}.
$$
Integrating from $0$ to $t$ and recalling that $R_{n}|_{t=0}\equiv 0$, we obtain
\begin{equation}\label{EQ_Lemme2_presque}
\partial_{t}R_{n} - \mu_{n}I_{n}^{0} =
d_{n} \Delta R_{n} +
\mu_{n} \int_{0}^{t} \alpha_{n} S_{n-1}I_{n} \, ds -
\mu_{n}R_{n}.
\end{equation}
Combining
\eqref{EQ_MODEL_c}
with the definition of $R_n$,
\eqref{EQ_RN_a},
the term under the integral rewrites
$$
\alpha_{n} S_{n-1}I_{n} = \partial_{t}\prth{R_{n-1}-S_{n-1}},
$$
so that we have
\begin{equation}\label{eq plus tard}
\partial_{t}R_{n} - \mu_{n}I_{n}^{0} =
d_{n} \Delta R_{n} +
\mu_{n} \prth{R_{n-1} - S_{n-1}} -
\mu_{n}R_{n}.
\end{equation}
Now, using the lower bound \eqref{EQ_le1_control_Sn} for $S_{n}$ given in \hyperref[LE_1_control_Sn_with_Rn]{Lemma \ref*{LE_1_control_Sn_with_Rn}}, we find
$$
\partial_{t}R_{n} \leq
d_{n}\Delta R_{n} +
\mu_{n} \Prth{R_{n-1}-R_{n-1} \, e^{-\frac{\alpha_{n}}{\mu_{n}}R_{n}}} -
\mu_{n}R_{n} +
\mu_{n}I_{n}^{0},
$$
which provides \eqref{EQ_Rn_sb_sol_KPP_perturbed}.
\end{proof}

\medskip

A consequence of \hyperref[LE_2_Rn_sb_s_KPP_perturbed]{Lemma \ref*{LE_2_Rn_sb_s_KPP_perturbed}} is that the functions $S_{n}, I_{n}, R_n$ are uniformly bounded.

\begin{lemmaNOBREAK}[Uniform upper bounds on $S_{n}$, $I_{n}$ and $R_{n}$]\label{lem bounded}
	For all $n \in \N$, there is $K_{n}>0$ such that, for all $t>0$ and $x\in\mathbb{R}$,
	$$
		S_n(t,x)+I_n(t,x)+R_n(t,x)\leq K_n.
	$$
\end{lemmaNOBREAK}

\begin{proof}[\texorpdfstring{Proof of \hyperref[lem bounded]{Lemma \ref*{lem bounded}}}{Proof of Lemma \ref*{lem bounded}}]

$\bullet$ \textit{Upper bound on $R_n$}. We argue by induction. If $n=1$, then as explained in \hyperref[n=1]{Remark \ref*{n=1}}, we have $\partial_{t} R_1 = d_{1} \Delta R_1 + f_1(R_1)+\mu_{1} I_{1}^{0}$, and $R_1(0,\point)\equiv0$. We can take $C_1>0$ such that $0\geq f_1(C_1)+ \mu_{1} I_{1}^{0}$, so that $C_1$ is supersolution to the \OK{equation} that $R_1$ solves. Hence, by the parabolic comparison principle, $R_1\leq C_1$.

\medskip

Now, for $n\in \N^\star$, $n>1$, assume that there is $C_{n-1}>0$ such that $R_{n-1} \leq C_{n-1}$. Hence, because $R_{n}$ satisfies \eqref{EQ_Rn_sb_sol_KPP_perturbed} from \hyperref[LE_2_Rn_sb_s_KPP_perturbed]{Lemma \ref*{LE_2_Rn_sb_s_KPP_perturbed}}, we have
\begin{align*}
	\partial_{t} R_{n}
	& \leq
	d_{n} \Delta R_{n}
	+ \mu_{n} C_{n-1}
		\prth{1-e^{-\frac{\alpha_{n}}{\mu_{n}}R_{n}}}
	- \mu_{n} R_{n}
	+ \mu_{n} I_{n}^{0}.
\end{align*}
Because $I_{n}^{0}$ is bounded, using the parabolic comparison principle, we can find then a constant $C_{n}>0$ large enough to be a supersolution of the above equation. This gives $R_{n} \leq C_{n}$ for all $t>0$ and $x\in\mathbb{R}$. By induction, each $R_n$ is therefore uniformly bounded.

\medskip

$\bullet$ \textit{Upper bound on $S_n$}.
Combining \hyperref[LE_1_control_Sn_with_Rn]{Lemma \ref*{LE_1_control_Sn_with_Rn}} and the previous point, we have $S_n\leq R_n\leq C_{n}$ for all $t>0$ and $x\in\mathbb{R}$.

\medskip

$\bullet$ \textit{Upper bound on $I_n$}.
For $n\in \N^{\star}$, we know that $I_{n}$ satisfies \eqref{EQ_MODEL_b}, which is a linear parabolic equation with bounded coefficients.
Therefore, owing to the parabolic Harnack inequality
---
see \cite[Section 7, Theorem 10]{EvansPartial10} for instance ---
there is a constant $k_{n}>0$ such that for any $t>1$ and any $x\in\mathbb{R}$, we have
$$
I_{n}(t,x) \leq k_n \inf_{\tau \in [t+1,t+2]}I_{n}(\tau,x).
$$
Therefore
$$
I_{n}(t,x) \leq k_n \int_{t+1}^{t+2} I_{n}(\tau,x) \, d\tau \leq \frac{k_n}{\mu_{n}}R_n(t+2,x)\leq \frac{k_n}{\mu_{n}}C_n,
$$
hence $I_n$ is also uniformly bounded.
\end{proof}

As explained above, we prove \hyperref[PROP_Rn_propagation]{Proposition \ref*{PROP_Rn_propagation}} by induction. The induction hypothesis will be denoted $\HR{n}$ for $n\in\intervalleE{1}{\NN}$ and is the following.

\vspace{6.5mm}
\renewcommand{\gap}{17mm}
\renewcommand{\gapp}{4mm}
\begin{systemB}{-14mm}{-14.2mm}{10mm}{l}{\label{HR}}{\HR{n} :}
	\sup_{\delta < \verti{x} <ct}
	\verti{R_{n}(t,x) - \Sa{n}}
	\underset{
		\delta , t \rightarrow \infty
		}{
		\loongrightarrow}
	0,
	\qquad\quad \forall c \in (0,c_n),\label{HR_a}\\[4mm]
	\sup_{\verti{x} > c t}
	\verti{R_{n}(t,x)} 
	\underset{
		\delta , t \rightarrow \infty
		}{
		\loongrightarrow}
	0,
	\hspace{24,351mm}\forall c \in \intervalleoo{c_{n}}{+\infty}.\label{HR_b}
\end{systemB}

\vspace{4mm}

As observed in \hyperref[n=1]{Remark \ref*{n=1}}, we already know that $\HR{n}$ holds true for $n=1$.

For the sake of clarity, it will useful in some places to adopt the convention that $\HR{0}$ is a vacuously true hypothesis, meaning it is always satisfied. In other words, when we state \glmt{assume that $\HR{n}$ holds for $n=0$}, nothing is actually being assumed.

\begin{remark}\label{REM_Rn_lower_than_Sn_ast_plus_eps_de_x}
	It is clear that $R_{n}(t,x)$ is non-decreasing with respect to $t$. For any $n\in \N^{\star}$, if $\HR{n}$ holds true, the boundedness of $R_{n}$ given by \hyperref[lem bounded]{Lemma \ref*{lem bounded}} tells us that there is $R_{n}^{\infty}\prth{x}$ such that
	$$
	R_{n}\prth{t,x} \underset{t \to \infty}{\nearrow} R_{n}^{\infty}\prth{x},
	$$
	and this convergence is \textnormal{a priori} only pointwise --- it turns out that it is actually locally uniform, we shall discuss this later.

	In addition, observe that taking the limit $t\to+\infty$ in \eqref{HR_a}$,\HR{n}$ also implies that $R^\infty_{n}$ converges to $\Sa{n}$ when $x$ is large, that is
	\begin{equation}\label{lim_rinfty}
		R_n^\infty(x)\underset{\vert x \vert \to +\infty}{\longrightarrow} S_{n}^{\star}.
	\end{equation}

	In particular --- we shall use this several time in the sequel --- this means that there is $\e_n(x)$ such that $\e_n(x)\underset{\vert x \vert \to +\infty}{\longrightarrow} 0$ and
	\begin{equation*}\label{EQ_RN_leq_Sast_plus_eps}
	R_{n}\prth{t,x} \leq \Sa{n} + \e_{n}\prth{x},
	\qquad
	\forall
	t>0,
	\, \forall
	x \in \mathbb{R}.
	\end{equation*}
\end{remark}

\subsection{\texorpdfstring{$R_{n}$}{Rn} solves a perturbed Fisher-KPP equation}\label{SS_Rn_perturbed_KPP}

We define the reaction function $f_{n}$ for $n\in\intervalleE{1}{\Nlast}$ as
\begin{equation}\label{EQ_def_fn}
f_{n}\prth{z} : = \mu_{n}\Prth{\Sa{n-1} \Prth{1 - e^{-\frac{\alpha_{n}}{\mu_{n}} z}} - z},
\qquad
\text{for any }z\in\mathbb{R}.
\end{equation}
The function $f_{n}$ is a KPP reaction, and $f_{n}'(0)=\alpha_n S_{n-1}^{\star} - \mu_{n}>0$ for $n\leq\NN$ (owing to our definition of $N^{\star}$).
Moreover observe that,
$\Sa{n}$ as defined in
\eqref{EQ_def_S_star}
is the unique positive zero of $f_{n}$.

\medskip

The main point of this subsection is to show that $R_{n}$ satisfies a \glmt{perturbed} KPP equation.
\begin{proposition}[$R_{n}$ solves a perturbed KPP equation]
\label{prop_Kpp_perturbe_full}

Let $n\in\intervalleE{1}{\Nlast}$ and assume that
$\HR{n-1}$
holds true.
Then, there is $\e_n \in L^\infty(\R)$ such that
$
	\varepsilon_{n}\prth{x}
	\underset{\verti{x}\to+\infty}{\loongrightarrow}
	0
$
and
\begin{equation}\label{Kpp_perturbe_full}
	-\e_n\leq\partial_{t} R_{n} -
	d_{n} \Delta R_{n} -
	f_{n}\prth{R_{n}}
	\leq
	\varepsilon_{n},
\end{equation}
for any $t>0$ and any $x\in\mathbb{R}$.
\end{proposition}

\begin{remark}\label{REM_explanation_KPP_perturbe}
\hyperref[prop_Kpp_perturbe_full]{Proposition \ref*{prop_Kpp_perturbe_full}} is the cornerstone for proving the spreading result stated in \hyperref[PROP_Rn_propagation]{Proposition \ref*{PROP_Rn_propagation}}. Indeed, if we had $\varepsilon_{n}=0$ in \eqref{Kpp_perturbe_full}
	---
	\ie if $R_{n}$ would solve exactly
	$\partial_{t} R_{n} = d_{n} \Delta R_{n} + f_{n}\prth{R_{n}}$
	---,
	classical results from reaction-diffusion theory would directly yield that $R_{n}$ spreads with speed $c_{n}$ toward $\Sa{n}$, that is, $\HR{n}$ would be verified.

	The presence of the perturbation $\varepsilon_{n}$ rises some technicalities, that will be tackled in the next \hyperref[SS_Propa_Rn]{Section \ref*{SS_Propa_Rn}}.
\end{remark}

The proof of \hyperref[prop_Kpp_perturbe_full]{Proposition \ref*{prop_Kpp_perturbe_full}} relies on \hyperref[LE_2_Rn_sb_s_KPP_perturbed]{Lemma \ref*{LE_2_Rn_sb_s_KPP_perturbed}} above and on the two other lemmas, namely \hyperref[LE_4_control_above_Rn_In]{Lemma \ref*{LE_4_control_above_Rn_In}} and \hyperref[LE_3_Rn_sp_s_KPP_perturbed]{Lemma \ref*{LE_3_Rn_sp_s_KPP_perturbed}}, that we now state and prove.

\begin{lemmaNOBREAK}[Exponential estimates]
	\label{LE_4_control_above_Rn_In}
	Let $n\in\intervalleE{1}{\Nlast}$
	and assume that the induction hypothesis $\HR{n-1}$ holds true.
	Then we have the following upper bounds on $R_{n}$ and $I_{n}$.

	For all $c>c_{n}$, there are $\Lambda_{n},\tilde{\Lambda}_{n}, \lambda_{n}>0$ such that
	\begin{equation}\label{LE_4_control_above_Rn_In_3}
		I_{n}\prth{t,x} \leq \Lambda_{n} \, e^{-\lambda_{n}\prth{x-ct}}, \qquad \forall t>0,\, \forall x\in \R,
	\end{equation}
	and
	\begin{equation}\label{LE_4_control_above_Rn_In_4}
		R_{n}\prth{t,x} \leq \tilde{\Lambda}_{n} \, e^{-\lambda_{n}\prth{x-ct}}, \qquad \forall t>0,\, \forall x\in \R.
	\end{equation}
\end{lemmaNOBREAK}

\begin{proof}[\texorpdfstring{Proof of \hyperref[LE_4_control_above_Rn_In]{Lemma \ref*{LE_4_control_above_Rn_In}}}{Proof of Lemma \ref*{LE_4_control_above_Rn_In}}]

$\bullet$ \textit{Proof of \eqref{LE_4_control_above_Rn_In_3}}.
Let $c>c_{n}$ be fixed. By definition of $c_n$, we can find $\e>0$ small enough so that
$$
c>2\sqrt{d_{n}(\alpha_n(S_{n-1}^{\star}+\e) - \mu_{n})}.
$$
From \eqref{EQ_le1_control_Sn}, we have
$$
\partial_{t} I_{n} = d_{n} \Delta I_{n} + (\alpha_n S_{n-1} - \mu_{n})I_{n}
\leq d_{n} \Delta I_{n} +(\alpha_n R_{n-1} - \mu_{n})I_{n},\qquad \forall t>0,\, \forall x\in \R. 
$$
Moreover, according to \hyperref[REM_Rn_lower_than_Sn_ast_plus_eps_de_x]{Remark \ref*{REM_Rn_lower_than_Sn_ast_plus_eps_de_x}}, since we assume that $\HR{n-1}$ holds, we can take $R>0$ large enough so that, for any $x>R$, we have $R_{n-1}(t,x)\leq S_{n-1}^{\star} + \e$ (where the $\e$ has been chosen above). As a result,
$$
\partial_{t} I_{n} \leq d_{n} \Delta I_{n} +(\alpha_n(S_{n-1}^{\star} +\e)-\mu_{n})I_{n},\qquad \forall t>0,\, \forall x>R.
$$
Define the function
$\bara{I_{n}}\prth{t,x} : = \Lambda_{n} e^{-\lambda_{n}\prth{x-ct}}$, where $\Lambda_{n}$ and $\lambda_n$ are two positive constants, chosen such that
$$
	c\lambda_{n} - d_{n}\lambda_{n}^{2} -
	\prthh{\alpha_{n} (\Sa{n-1} + \e) - \mu_{n}}
	\geq 0,
$$
(this is possible because the discriminant of this equation is positive), and with $\Lambda_{n}$ chosen sufficiently large to ensure that, at initial time,
$$
	\bara{I_{n}}|_{t=0}\prth{x}
	\, = \,
	\Lambda_{n} e^{-\lambda_{n}x}
	\, \geq \,
	I_{n}|_{t=0}\prth{x},
	\qquad
	\text{for all }x>R,
$$
and that, for all $t>0$,
\begin{equation}\label{EQ_assert_I_{n}_super_sol}
	\bara{I_{n}}\prth{t,R}
	\, = \,
	\Lambda_{n} e^{-\lambda_{n}\prth{R-ct}}
	\, \geq \,
	\Lambda_{n} e^{-\lambda_{n}R}
	\, \geq \,
	\sup_{t>0}\sup_{x\in \R}I_n(t,x).
\end{equation}
Therefore, the function $\bara{I_{n}}$ satisfies $\partial_{t} \bara{I_{n}} \geq d_{n}\Delta \bara{I_{n}} +(\alpha_n(S_{n-1}^{\star}+\e)-\mu_{n})\bara{I_{n}}$ for $t>0$ and $x>R$, and $\bara{I_{n}}(0,\point)\geq I_{n}(0,\point)$ and $\bara{I_{n}}(t,R)\geq I_{n}(t,R)$. Hence, the parabolic comparison principle (applied on the set $(0,+\infty)\times (R,+\infty)$ to $\bara{I_{n}}, I_{n}$) yields $\bara{I_{n}} \geq I_{n}$
for any $t>0$ and $x>R$.

Now, for $x\leq R$ and $t>0$, we have, using \hyperref[lem bounded]{Lemma \ref*{lem bounded}},
$$
	\bara{I_{n}}\prth{t,x}
	\, \geq \,
	\Lambda_{n} e^{-\lambda_{n}R}
	\stackrel{\eqref{EQ_assert_I_{n}_super_sol}}{\geq}
	\sup_{t>0}\sup_{x\in \R}I_{n}\prth{t,x},
$$
and this completes the proof of \eqref{LE_4_control_above_Rn_In_3}.

\medskip

$\bullet$ \textit{Proof of \eqref{LE_4_control_above_Rn_In_4}}.
In view of \eqref{EQ_RN}, we multiply \eqref{LE_4_control_above_Rn_In_3} by $\mu_{n}$ before integrating between $0$ to $t$, this yields
$$
	R_{n}\prth{t,x} \leq
	\frac{\mu_{n}}{c\lambda_{n}} \Lambda_{n} e^{-\lambda_{n}\prth{x-ct}},
$$
for all $t>0$ and $x\in\mathbb{R}$. This gives \eqref{LE_4_control_above_Rn_In_4}, with
$\tilde{\Lambda}_{n}=\frac{\mu_{n}}{c\lambda_{n}} \Lambda_{n}$.
\end{proof}

\begin{remark}[Symmetric version of \eqref{LE_4_control_above_Rn_In_3} and \eqref{LE_4_control_above_Rn_In_4}]\label{REM_expo_upper_bounds_LHS}
	Using similar arguments, we also find that, for all $t>0$ and $x\in \R$,
	\eqref{LE_4_control_above_Rn_In_3} and \eqref{LE_4_control_above_Rn_In_4},
	\begin{equation}\label{LE_4_control_above_Rn_In_3_sym}
		I_{n}\leq \Lambda_{n} \, e^{\lambda_{n}\prth{x+ct}},
	\end{equation}
	and
	\begin{equation}\label{LE_4_control_above_Rn_In_4_sym}
		R_{n} \leq \tilde{\Lambda}_{n} \, e^{\lambda_{n}\prth{x+ct}}.
	\end{equation}
\end{remark}

\begin{lemmaNOBREAK}[$R_{n}$ is super-solution to a perturbed Fisher-KPP equation]\label{LE_3_Rn_sp_s_KPP_perturbed}
	Assume $\Nlast\geq 2$. Let $n\in\intervalleE{2}{\Nlast}$
	and assume that the induction hypothesis $\HR{n-1}$ holds true.
	Then there is $\e_n \in L^\infty(\R)$ such that	$\varepsilon_{n}\prth{x} \underset{\verti{x}\to+\infty}{\loongrightarrow} 0$
	and
	\begin{equation}\label{EQ_Rn_sp_sol_KPP_pertur_1}
		\partial_{t}R_{n} \geq
		d_{n}\Delta R_{n} +
		f_{n}\prth{R_{n}} -
		\varepsilon_{n},\qquad \forall t>0,\, \forall x\in \R.
	\end{equation}
\end{lemmaNOBREAK}

\begin{proof}[\texorpdfstring{Proof of \hyperref[LE_3_Rn_sp_s_KPP_perturbed]{Lemma \ref*{LE_3_Rn_sp_s_KPP_perturbed}}}{Proof of Lemma \ref*{LE_3_Rn_sp_s_KPP_perturbed}}]
Consider \eqref{EQ_Lemme2_presque} as established in the proof of \hyperref[LE_2_Rn_sb_s_KPP_perturbed]{Lemma \ref*{LE_2_Rn_sb_s_KPP_perturbed}}.
By using the definition of $R_{n}$ \eqref{EQ_RN_a}, this equality can be recast
\begin{equation}\label{EQ_Lemme3_presque}
	\partial_{t}R_{n} =
	d_{n} \Delta R_{n}
	+ \mu_{n} \int_{0}^{t} \frac{\alpha_{n}}{\mu_{n}} S_{n-1} \, \partial_{t}R_{n} \, ds
	- \mu_{n}R_{n}
	+ \mu_{n}I_{n}^{0}.
\end{equation}
Let $\tilde{c} =  \frac{c_n + c_{n-1}}{2}$ be fixed. Due to the positivity of the integrated term in \eqref{EQ_Lemme3_presque} we can write (where the minimum of two reals $a$ and $b$ is denoted $a \wedge b :=\min\{a,b\}$)
$$
	\int_{0}^{t} \frac{\alpha_{n}}{\mu_{n}} S_{n-1} \, \partial_{t}R_{n} \, ds
	\geq
	\int_{\MIN}^{t} \; \frac{\alpha_{n}}{\mu_{n}} S_{n-1} \, \partial_{t}R_{n} \, ds,\qquad \forall t>0,\ x\in \R.
$$
Using \eqref{EQ_le1_control_Sn} from \hyperref[LE_1_control_Sn_with_Rn]{Lemma \ref*{LE_1_control_Sn_with_Rn}}, we get
$$
	\int_{0}^{t} \frac{\alpha_{n}}{\mu_{n}} S_{n-1} \, \partial_{t}R_{n} \, ds
	\geq
	\int_{\MIN}^{t} \; \frac{\alpha_{n}}{\mu_{n}} R_{n-1} \, e^{-\frac{\alpha_{n}}{\mu_{n}}R_{n} } \, \partial_{t}R_{n} \, ds,
$$
and using the fact that $R_{n-1}(t,x)$ is non-decreasing with respect to the $t$ variable yields
\renewcommand{\gap}{-25mm}
\begin{align}
 	\int_{0}^{t} \frac{\alpha_{n}}{\mu_{n}} S_{n-1} \, \partial_{t}R_{n} \, ds&
	\nonumber\\[3mm]
	& \hspace{\gap}
	\geq R_{n-1}\!\Prth{\MIN,x}
	\int_{\MIN}^{t} \; \frac{\alpha_{n}}{\mu_{n}} e^{-\frac{\alpha_{n}}{\mu_{n}}R_{n}} \, \partial_{t}R_{n} \, ds \nonumber\\[3mm]
	& \hspace{\gap}
	= R_{n-1}\!\Prth{\MIN,x}
	\int_{\MIN}^{t} \; - \partial_{t} \Prth{e^{-\frac{\alpha_{n}}{\mu_{n}}R_{n}}} \, ds
	\nonumber\\[3mm]
	& \hspace{\gap}
	= R_{n-1}\!\Prth{\MIN,x}
	\Croch{
		\exp\Prth{-\frac{\alpha_{n}}{\mu_{n}}R_{n}\!\Prth{\MIN,x}} -
		\exp\Prth{-\frac{\alpha_{n}}{\mu_{n}}R_{n}\prth{t,x}}
	}\nonumber\\[4mm]
	& \hspace{\gap}
	= R_{n-1}\!\Prth{\MIN,x}
	\Croch{
		\Prth{1 - \exp\Prth{-\frac{\alpha_{n}}{\mu_{n}}R_{n}\prth{t,x}}} -
		\Prth{1 - \exp\Prth{-\frac{\alpha_{n}}{\mu_{n}}R_{n}\!\Prth{\MIN,x}}}
	}\nonumber\\[2mm]
	& \hspace{-15mm}
	+ \Sa{n-1}\Prth{1 - \exp\Prth{-\frac{\alpha_{n}}{\mu_{n}}R_{n}\prth{t,x}}}
	- \Sa{n-1}\Prth{1 - \exp\Prth{-\frac{\alpha_{n}}{\mu_{n}}R_{n}\prth{t,x}}}\nonumber\\[4mm]
	& \hspace{\gap}
	= \Sa{n-1}\Prth{1 - \exp\Prth{-\frac{\alpha_{n}}{\mu_{n}}R_{n}\prth{t,x}}}
	- \rho^1_{n}\prth{t,x}-\rho^2_{n}\prth{t,x},\label{EQ_Lemme3_presque_bis}
\end{align}
where
\begin{equation}\label{EQ_E_erreur_dessous1}
    \rho^1_{n}\prth{t,x} = R_{n-1}\!\Prth{\MIN, x}
		\Croch{
			1-\exp\!\Prth{-\frac{\alpha_{n}}{\mu_{n}}R_{n}\!\Prth{\MIN, x}}}
\end{equation}
and
\begin{equation}\label{EQ_E_erreur_dessous2}
    \rho^2_{n}\prth{t,x} = \crochhhh{\Sa{n-1} - R_{n-1}\!\Prth{\MIN, x}}
		\Croch{
			1-\exp\!\Prth{-\frac{\alpha_{n}}{\mu_{n}}R_{n}\prth{t, x}}}.
\end{equation}
Combining \eqref{EQ_Lemme3_presque} and \eqref{EQ_Lemme3_presque_bis},
we are eventually led to
\begin{equation}\label{EQ_Lemme3_presque_ter}
	\partial_{t}R_{n} \geq
	d_{n}\Delta R_{n} +
	f_{n}\prth{R_{n}} +
	  \mu_{n} I_{n}^{0}
	- \mu_{n} \rho_{n}^1\prth{t,x} - \mu_{n} \rho^2_{n}\prth{t,x}.
\end{equation}
The goal now is to draw estimates from above of
$\rho_{n}^{1}\prth{t,x}$ and $\rho_{n}^{2}\prth{t,x}$.

\medskip

$\bullet$ \textit{Estimate on $\rho_{n}^{1}$}. Owing to \hyperref[lem bounded]{Lemma \ref*{lem bounded}}, $R_{n-1} \leq K_{n-1}$. Using the concavity of $z\mapsto 1-e^{-z}$ and the fact that
$R_{n}$ is non-decreasing with respect to the $t$ variable, we have
\begin{align*}
	\rho_{n}^{1}\prth{t,x}
	&\leq K_{n-1}
		\Croch{
			1-\exp\!\Prth{-\frac{\alpha_{n}}{\mu_{n}}R_{n}\!\Prth{\MIN, x}}
		}
	\\[2mm]
	&\leq K_{n-1}\frac{\alpha_{n}}{\mu_{n}}
	R_{n}\!\Prth{\MIN, x}
	\\[2mm]
	&\leq K_{n-1}\frac{\alpha_{n}}{\mu_{n}}
	R_{n}\!\Prth{\frac{\verti{x}}{\tilde{c}}, x}.
 \end{align*}
Now, fixing $c>0$ such that $c_{n}<c<\tilde{c}=\prth{c_{n}+c_{n-1}}/2$, \eqref{LE_4_control_above_Rn_In_4} from \hyperref[LE_4_control_above_Rn_In]{Lemma \ref*{LE_4_control_above_Rn_In}} gives (for $x>0$)
$$
	\rho_{n}^{1}\prth{t,x}
	\leq
	K_{n-1}\frac{\alpha_{n}}{\mu_{n}}
	\tilde{\Lambda}_{n} \times
	\text{exp}\prthhhh{
		\hspace{-1mm} -\lambda_{n}\underbrace{\Prth{1-\frac{c}{\tilde{c}}}}_{>0}x
	}.
$$
 Similarly, we get, for $x<0$,
$$
	\rho_{n}^{1}\prth{t,x}
	\leq
	K_{n-1}\frac{\alpha_{n}}{\mu_{n}}
	\tilde{\Lambda}_{n} \times
	\text{exp}\prthhhh{
		\lambda_{n}\Prth{1-\frac{c}{\tilde{c}}}x
	},
$$
so that in the end, we have that there are $K,q>0$ such that
\begin{equation}\label{EQ_estimate_rho_1}
	\rho_{n}^{1}\prth{t,x}
	\leq
	Ke^{-q \vert x \vert}.
\end{equation}

\medskip

$\bullet$ \textit{Estimate on $\rho_{n}^{2}$}. 
We write
$$
\rho^2_{n}\prth{t,x} = \rho^2_{n}\prth{t,x}\indicatrice{\vert x \vert < \tilde c t} + \rho^2_{n}\prth{t,x}\indicatrice{\vert x \vert \geq \tilde c t}.
$$
Because $S_{n},R_n$ are bounded (thanks to \hyperref[LE_4_control_above_Rn_In]{Lemma \ref*{LE_4_control_above_Rn_In}}), and using the concavity of $z\mapsto 1- e^{-z}$ and the fact that $R_n$ is non-decreasing with respect to the $t$ variable, we have that there is $K>0$ such that
$$
\vert \rho^2_{n}\prth{t,x}\indicatrice{\vert x \vert \geq \tilde c t}\vert \leq K R_n(t,x)\indicatrice{\vert x \vert \geq \tilde c t}\leq K R_n\left(\frac{\vert x \vert}{\tilde c},x\right).
$$
We are in the same situation than in the previous step, we can use the exponential estimates from \hyperref[LE_4_control_above_Rn_In]{Lemma \ref*{LE_4_control_above_Rn_In}} to find that there is $q>0$ such that
$$
\vert \rho^2_{n}\prth{t,x}\indicatrice{\vert x \vert \geq \tilde c t}\vert\leq Ke^{-q \vert x \vert}.
$$
On the other hand, we have
$$
\vert \rho^2_{n}\prth{t,x}\indicatrice{\vert x \vert < \tilde c t}\vert \leq \left\vert S_{n-1}^{\star} - R_{n-1}\Prth{\MIN,x}\right\vert \indicatrice{\vert x \vert < \tilde c t}
\leq \left\vert S_{n-1}^{\star} - R_{n-1}\Prth{\frac{\vert x \vert}{\tilde c},x}\right\vert .
$$
Because $\HR{n-1}$ holds true and because $\tilde c<c_{n-1}$, the quantity $x\mapsto \vert S_{n-1}^{\star} - R_{n-1}(\frac{\vert x \vert}{\tilde c},x)\vert$ goes to zero as $\vert x \vert$ goes to $+\infty$.

\medskip

Therefore, we have proven that $\rho_n^1(t,x)+\rho_n^2(t,x) \leq \e(x)$, where $\e(x)\underset{\vert x \vert \to +\infty}{\longrightarrow}0$, the result follows.
\end{proof}

\medskip

Combining all that precedes, we are now in position to prove \hyperref[prop_Kpp_perturbe_full]{Proposition \ref*{prop_Kpp_perturbe_full}}.

\medskip

\begin{proof}[\texorpdfstring{Proof of \hyperref[prop_Kpp_perturbe_full]{Proposition \ref*{prop_Kpp_perturbe_full}}}{Proof of Proposition \ref*{prop_Kpp_perturbe_full}}]
If $n=1$, then we already have the result as explained in \hyperref[n=1]{Remark~\ref*{n=1}}. We suppose from now on that $\Nlast\geq 2$ and $n\geq 2$.

Assume that $\HR{n-1}$ holds true. Then, owing to \hyperref[REM_Rn_lower_than_Sn_ast_plus_eps_de_x]{Remark \ref*{REM_Rn_lower_than_Sn_ast_plus_eps_de_x}}, we have $R_{n-1}\leq S_{n-1}^{\star} + \e_n(x)$, for some function $\e_n$ such that $\e_n(x)\underset{\vert x \vert \to +\infty}{\longrightarrow}0$. Combining this with \hyperref[LE_2_Rn_sb_s_KPP_perturbed]{Lemma \ref*{LE_2_Rn_sb_s_KPP_perturbed}} yields that
$$
\partial_{t}R_{n} \leq
	d_{n}\Delta R_{n} +
	\mu_{n} S_{n-1}^{\star}\!\Prth{1-e^{-\frac{\alpha_{n}}{\mu_{n}}R_{n}}} -
	\mu_{n}R_{n}
	+\mu_{n} \e_n + \mu_{n}I_{n}^{0}.
$$
On the other hand, \hyperref[LE_3_Rn_sp_s_KPP_perturbed]{Lemma \ref*{LE_3_Rn_sp_s_KPP_perturbed}} yields that there is $\tilde \varepsilon_{n}$ such that $\tilde \varepsilon_{n}\prth{x} \underset{\verti{x}\to+\infty}{\loongrightarrow} 0$ so that
	\begin{equation}\label{EQ_Rn_sp_sol_KPP_pertur_2}
		\partial_{t}R_{n} \geq
		d_{n}\Delta R_{n} +
		f_{n}\prth{R_{n}} -
		\tilde \varepsilon_{n}.
	\end{equation} 
Therefore, up to renaming $\max\{\vert \tilde \e_n\vert,\mu_{n}\vert \e_n\vert+\mu_{n}I_{n}^{0}\}$ as $\e_n$, the result follows.
\end{proof}

\subsection{Propagation of \texorpdfstring{$R_{n}$}{Rn} for \texorpdfstring{$n\in \intervalleE{1}{\Nlast}$}{1<=n<=N}}\label{SS_Propa_Rn}

The aim of this section is to prove \hyperref[PROP_Rn_propagation]{Proposition \ref*{PROP_Rn_propagation}}. We will proceed by induction: assuming $\HR{n-1}$, we prove that $\HR{n}$ holds true.

\OK{As a first step, we prove \eqref{HR_a} in a weaker form, in the sense that we show the convergence without the speed. This is outlined in the following lemma.}

\begin{lemma}\label{LE_5}
Let $n \in \intervalleE{1}{\Nlast}$ and assume that $\HR{n-1}$ holds true. 
Then
\begin{equation}\label{LE_5_a}
R_{n}\prth{t,x} \underset{t\to+\infty}{\longrightarrow} R_{n}^{\infty}\prth{x},
\end{equation}
locally uniformly.
Moreover
\begin{equation}\label{LE_5_b}
R_{n}^\infty\prth{x} \underset{\verti{x} \to +\infty}{\longrightarrow} \Sa{n}.
\end{equation}
\end{lemma}
\OK{Observe that \hyperref[LE_5]{Lemma \ref*{LE_5}} differs from \hyperref[REM_Rn_lower_than_Sn_ast_plus_eps_de_x]{Remark \ref*{REM_Rn_lower_than_Sn_ast_plus_eps_de_x}}, since we claim \eqref{LE_5_a}-\eqref{LE_5_b} under $\HR{n-1}$, and not $\HR{n}$.}

We also recall that $\HR{0}$ is vacuously true, that is, when $n=1$, there is no hypothesis in the lemma, and we already know that the result is true, as explained in \hyperref[n=1]{Remark \ref*{n=1}}.

\medskip

\begin{proof}[\texorpdfstring{Proof of \hyperref[LE_5]{Lemma \ref*{LE_5}}}{Proof of Lemma \ref*{LE_5}}]
\OK{Assume that $\HR{n-1}$ holds true.}

\medskip
$\bullet$ \textit{Step 1. $R_{n}$ converges locally uniformly.}
Because the function $R_{n}(t,x)$ is non-decreasing with respect to $t$, and because it is uniformly bounded --- thanks to \hyperref[LE_4_control_above_Rn_In]{Lemma \ref*{LE_4_control_above_Rn_In}} ---, there is $R_\infty\prth{x}$ such that $R_{n}(t,x) \underset{t\to+\infty}{ \longrightarrow} R_{n}^\infty\prth{x}$. This convergence is pointwise. However, $R_{n}(t,x)$ solves \eqref{prop_Kpp_perturbe_full}, therefore, owing to standard parabolic regularity theory \cite{LiebermanSecond96}, the convergence is actually locally uniform, and moreover the limit function $R_{n}^\infty\prth{x}$ solves
$$
-\e_{n} \leq -d_{n}\Delta R_{n}^\infty - f_{n}(R_{n}^\infty) - \mu_{n} I^{0}_{n} \leq \e_{n}.
$$
In the rest of the proof, we take $\rho\in\R$ to be an arbitrary limit point of $R_{n}^\infty\prth{x}$ when $\verti{x}$ goes to $+\infty$.
This means that there is a sequence $(x_k)_{k\in \N}$ such that $\vert x_k\vert \to +\infty$ and
$$
\rho = \lim_{k\to+\infty}R_{n}^\infty(x_k).
$$
The rest of the proof consists in showing that $\rho = \Sa{n}$.

\medskip

$\bullet$ \textit{Step 2. Either $\rho=0$ or $\rho = \Sa{n}$.}
We denote
$$
\rho_k\prth{x} := R_{n}^\infty(x+x_k).
$$
It is classical from elliptic regularity theory \cite{GilbargElliptic01} that, up to extraction,
$$
\rho_k\prth{x}\underset{k\to+\infty}{\longrightarrow} \rho_\infty\prth{x},
$$
where $\rho_\infty$ solves the elliptic equation
$$
-d_{n}\Delta \rho_\infty - f_{n}(\rho_\infty) = 0.
$$
Moreover, because $R_{n}$ is uniformly bounded, the same holds true for $\rho_\infty$. It is a classical result in reaction-diffusion equations theory that the only bounded solutions of this equation are the zeros of the function $f_{n}$, that are the constants $0$ and $\Sa{n}$ (see \cite{AronsonMultidimensional78} for instance).

\medskip

By definition of $\rho_k$, we have
$$
\rho = \lim_{k\to+\infty}\rho_k(0),
$$
therefore, either $\rho=0$ or $\rho = \Sa{n}$.

\medskip

$\bullet$ \textit{Step 3. Proof that $\rho>0$.}
Assume by contradiction that $\rho=0$. Arguing as in the previous step, this implies that $R_{n}^\infty(\point+x_k) \to 0$ locally uniformly as $k \to +\infty$.

Now, let us show that there are $k\in \N$ and $T>0$ large enough so that
\begin{equation}\label{sur eps}
    \partial_{t} I_{n} \geq d_{n}\Delta I_{n} + (\alpha_{n} (S_{n-1}^{\star}-\e) e^{-\frac{\alpha_{n}}{\mu_{n}}\e} - \mu_{n}) I_{n},\qquad \forall t>T,\, \forall x\in B_R(-x_k).
\end{equation}
Because $R_{n}^\infty(\point+x_k) \to 0$ locally uniformly as $k$ goes to $+\infty$, we can find $k$ large enough so that
$$
R_{n}^{\infty}(x+x_k)\leq \e, \qquad \forall x\in B_R(0).
$$
Therefore, thanks to \eqref{EQ_le1_control_Sn}, we have, for any $t>0$ and $x\in B_R(-x_k)$,
$$
S_{n-1}(t,x)\geq R_{n-1}(t,x)e^{-\frac{\alpha_{n}}{\mu_{n}}R_{n}(t,x)}\geq R_{n-1}(t,x)e^{-\frac{\alpha_{n}}{\mu_{n}}R_{n}^\infty\prth{x}}
\geq R_{n-1}(t,x)e^{-\frac{\alpha_{n}}{\mu_{n}}\e}.
$$
Hence, because $\HR{n-1}$ holds true, up to increasing $k$ if needed, and choosing $T>0$ large enough, we have, for all $t>T$ and $x \in B_R(-x_k)$,
$$
S_{n-1}(t,x) \geq (S_{n-1}^{\star}-\e) e^{-\frac{\alpha_{n}}{\mu_{n}}\e}.
$$
This latter inequality implies that \eqref{sur eps} holds true.

\medskip

Let now $\lambda$ be the principal eigenvalue of the operator $-d_{n}\Delta  -(\alpha_n(S_{n}^{\star} - \e)e^{-\frac{\alpha_n}{\mu_{n}}\e}-\mu_{n})$ on $B_R(-x_k)$ with Dirichlet boundary conditions, and let $\phi$ be a positive eigenfunction associated to the principal eigenvalue (its existence is guaranteed by the Krein-Rutman theorem \cite{KreinLinear48}. Therefore, $\phi \in C^2(B_R(-x_k))$ is such that $\phi>0$ on $B_R(-x_k)$, $\phi=0$ on $\partial B_R(-x_k)$ and
$$
-d_{n} \Delta \phi -(\alpha_n(S_{n}^{\star} - \e)e^{-\frac{\alpha_n}{\mu_{n}}\e}-\mu_{n})\phi = \lambda \phi.
$$
Let $v(t,x) = \phi(x)e^{-\lambda t}$, it solves
$$
\partial_{t} v = d_{n}\Delta v + (\alpha_{n} (S_{n-1}^{\star}-\e) e^{-\frac{\alpha_{n}}{\mu_{n}}\e} - \mu_{n}) v,\qquad t>T,\ x\in B_R(-x_k)
$$
with Dirichlet boundary condition $v(t,x)=0$ for $x\in \partial B_R(-x_k)$ and $t>T$, and up to multiplying $\phi$ by a small constant, we can ensure that $v(T,\point)\leq I_{n}(T,\point)$. It then follows from the parabolic comparison principle that $I_{n}(t,x)\geq v(t,x)$ for all $t>T$ and $x\in B_R(-x_k)$.
It is classical (see \cite{BerestyckiGeneralizations15}) that, up to taking $R$ large enough, the principal eigenvalue $\lambda$ can be made as close as we want to $-(\alpha_n(S_{n}^{\star} - \e)e^{-\frac{\alpha_n}{\mu_{n}}\e}-\mu_{n})$, which is strictly negative (because $n\leq \Nlast$).

\medskip

Therefore, \OK{$v(t,x)=\phi(x)e^{-\lambda t}\to+\infty$ for all $x$ in $B_R(-x_k)$ as $t\to+\infty$}, which is in contradiction with the boundedness of $I_{n}$ given by \hyperref[lem bounded]{Lemma \ref*{lem bounded}}.

\medskip

In conclusion, we have proved that $\rho = S^{\star}_{n}$, that is, $R^\infty_{n}\prth{x}$ indeed converges toward $\Sa{n}$ as $\vert x\vert$ goes to $+\infty$.
\end{proof}

\medskip

We now turn to the proof of \hyperref[PROP_Rn_propagation]{Proposition \ref*{PROP_Rn_propagation}}. The idea is to compare $R_n$ with a function solution of a {\it bistable} equation. Bistable reaction-diffusion equations are PDE of the form
$$
\partial_{t} u = d\Delta u +f(u),
$$
where the function $f$ vanishes at three points: $0<\theta<\sigma$ and is such that $f<0$ on $(0,\theta)$ and $f>0$ on $(\sigma,\theta)$. There is a wide literature on this specific type of equations \cite{AronsonMultidimensional78, BerestyckiFront02, DucasseBlocking18, RossiFreidlin17}. In particular, we recall the two following technical results that we shall need:
\begin{propositionNOBREAK}[Sufficient condition for spreading in bistable equations]\label{prop_bistable}
Let $f$ be a Lipschitz continuous function such that there are $0<\theta<\sigma$ such that $f(0)=f(\theta)=f(\sigma)=0$ and $f<0$ on $(0,\theta)$ and $f>0$ on $(\theta,\sigma)$. Assume that $\int_0^\sigma f(x)dx>0$.

\medskip

Let $u$ be the solution to
\begin{equation}\label{EQ_bistable}
    \partial_{t} u = d\Delta u + f(u),\qquad t>0,\ x>0,
\end{equation}
with Dirichlet boundary condition $u(t,0)=0$ for $t>0$ and with initial datum $u_0\geq 0$ compactly supported.

Then, there is $c^\star>0$ such that, for all $\e>0$, there is $L>0$ such that if $u_0\geq (\theta+\e)\indicatrice{\intervalleff{0}{L}}$, the function $u$ spreads toward $\sigma$ with speed $c^{\star}>0$ (which depends only on $f$) in the sense that
$$
\sup_{\delta< x < ct} \vert u\OK{\prth{t,x}} -\sigma\vert \underset{\delta,t\to+\infty}{\longrightarrow}0,\qquad \forall c\in (0,c^{\star}),
$$
and
$$
\sup_{x > ct} \vert u\OK{\prth{t,x}} \vert \underset{t\to+\infty}{\longrightarrow}0,\qquad \forall c> c^{\star}.
$$
\end{propositionNOBREAK}
This result is classical and is based on constructing appropriate compactly supported subsolutions to \eqref{EQ_bistable}. We refer to \cite{AronsonMultidimensional78, DucasseBlocking18} for proofs of this facts.

The next proposition tells us that, when we have a family of bistable nonlinearities that converge to a KPP nonlinearity, then the spreading speeds associated also converge. We shall use it with our KPP nonlinearities $f_n$ defined in \eqref{EQ_def_fn}.
\begin{proposition}[Convergence of bistable speeds]\label{prop_cv_speed}
Let $n\in \intervalleE{1}{\Nlast}$ and consider the KPP nonlinearity $f_n$ defined in \eqref{EQ_def_fn}.

Let $(f_\eta)_{\eta>0}$ be a family of bistable nonlinearities such that $f_\eta \to f_n$ locally uniformly as $\eta\to 0$.
For each $\eta>0$, let us denote $c_\eta$ the corresponding speed of spreading given by \hyperref[prop_bistable]{Proposition \ref*{prop_bistable}}. Then
$$
c_\eta \underset{\eta\to 0}{\longrightarrow}c_n.
$$
\end{proposition}

Unlike the speed of propagation for KPP reaction-diffusion equations, there is no explicit formula for the speed of propagation for bistable equations. There are variational formulas (see \cite{HamelFormules99} for instance), but the proof of \hyperref[prop_cv_speed]{Proposition \ref*{prop_cv_speed}} can be obtained directly by taking the limit of the bistable traveling fronts for each $f_\eta$, see \cite[Proposition 2.6]{RossiFreidlin17} (we also refer to \cite{BerestyckiFront02} for a different approach).

\medskip

{We are now in position to prove \hyperref[PROP_Rn_propagation]{Proposition \ref*{PROP_Rn_propagation}}, which establishes the propagation of $R_n$.}

\medskip

\begin{proof}[\texorpdfstring{Proof of \hyperref[PROP_Rn_propagation]{Proposition \ref*{PROP_Rn_propagation}}}{Proof of Proposition \ref*{PROP_Rn_propagation}}]
The proof is done by induction. Let $n\in \intervalleE{1}{\Nlast}$ and assume that $\HR{n-1}$ holds true. Let us show that $\HR{n}$ also holds true.

\medskip

Let $\eta, L>0$ to be chosen small enough and large enough after.

\medskip

Because we assume that $\HR{n-1}$ holds true, \hyperref[prop_Kpp_perturbe_full]{Proposition \ref*{prop_Kpp_perturbe_full}} tells us that  $R_{n}$ solves \eqref{Kpp_perturbe_full} (the perturbed KPP equation), where $\e_{n}\prth{x}$ goes to zero as $\verti{x}$ goes to $+\infty$.

\medskip

Therefore, we can take $R>0$ such that $\vert \e_{n}\prth{x}\vert \leq \eta$ for $x >R$, so that the function $R_{n}$ satisfies
\begin{equation}\label{eq R eta}
	\partial_{t} R_{n} \geq d_{n} \Delta R_{n} + f_{n}(R_{n}) -\eta,
	\qquad t>0, \, x>R.
\end{equation}
We now prove that the function $R_n$ spreads toward $S_{n}^{\star}$ with the wanted speed toward the right (when $x\to+\infty$), the spreading toward the left (for $x\to -\infty$) can be done similarly.

Owing to \hyperref[LE_5]{Lemma \ref*{LE_5}}, up to increasing $R$ if needed, we can find $T>0$ large enough so that
$$
R_{n}(t,x)\geq \Sa{n}-\eta,\qquad \forall t\geq T, \, \forall x\in [R,R+L].
$$
We define $\rho(t,x) := R_n(t,x)+\eta$. It satisfies the three following properties
\renewcommand{\gap}{10mm}
\renewcommand{\gapp}{2mm}
\begin{equation*}
\left\lbrace
\begin{array}{llll}
	\rho(t,x)\geq \eta,
	&& \forall t>T, \, \forall x>R,\\[1mm]
	\rho(T,x) \geq S_{n}^{\star},
	&& \forall x\in [R,R+L],\\[1mm]
	\partial_{t} \rho \geq d_{n} \Delta \rho + f_\eta(\rho),
	&& t>T, \, x>R,
\end{array}
\right .
\end{equation*}
where (see \hyperref[FIG_schema_graph]{Figure \ref*{FIG_schema_graph}} below),
$$
	f_{\eta}(v) := \begin{cases}
        f_{n}(v-\eta) -\eta,\quad &\text{if}\quad v\geq\eta,\\
		-v,\quad &\text{if} \quad v \in (0,\eta).
    \end{cases}
$$
Observe that in the last inequality we can use $f_\eta$ as defined and not simply $f_n(\point - \eta) - \eta$, because the function $\rho$ is always larger that $\eta$, therefore, the values taken by $f_\eta(v)$ for $v<\eta$ do not matter.

The function $f_\eta$ is a bistable nonlinearity in the sense defined above, provided $\eta$ is sufficiently small. The function $f_\eta$ vanishes at $x=0$ and at two other points $0<\theta< \sigma$. Observe that, as $\eta$ goes to $0^+$, we have $\theta\to0$ and $\sigma\to S_{n}^{\star}$. The graph of $f_\eta$ is depicted below.

\refstepcounter{FIGURE}\label{FIG_schema_graph}
\begin{center}
\includegraphics[scale=1]{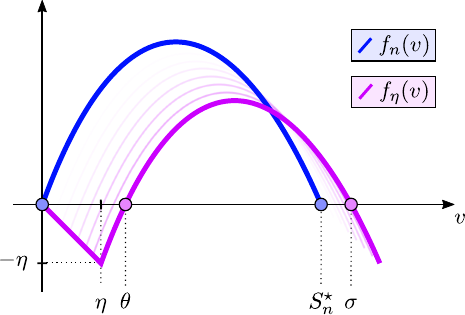}\\[1mm]
\begin{minipage}[c]{145mm}
\begin{footnotesize}
\begin{center}
\textsl{\textbf{Figure \theFIGURE} ---
	Representation of the functions $f_{n}$ and $f_{\eta}$. 
}
\end{center}
\end{footnotesize}
\end{minipage}
\end{center}
We take $\eta$ small enough so that $\theta<S_{n}^{\star}$ and we denote $v$ the solution to the following equation
\begin{equation}\label{eq bis}
\partial_{t} v = d_{n}\Delta v + f_\eta(v),\qquad t>T, \,  x  >R,
\end{equation}
with \glmt{initial} datum {$v(T,x) = S^\star_{n} \indicatrice{\intervalleff{R}{R+L}}\prth{x}$} and Dirichlet boundary condition $v(t,R) = 0$ for all $t>T$.

Thanks to \hyperref[prop_bistable]{Proposition \ref*{prop_bistable}}, we can choose $L$ large enough so that the solution $v$ of \eqref{eq bis} spreads toward $\sigma$ with some speed that we denote $c_\eta$ (to emphasize the dependence on $\eta$).

\medskip

Let us take $c\in (0,c_n)$. Up to taking $\eta$ small enough, owing to \hyperref[prop_cv_speed]{Proposition \ref*{prop_cv_speed}}, we can ensure $c_\eta > c$.

The parabolic comparison principle then implies that $\rho\geq v$ for any $t>T$ and $x>R$. Therefore, because $\rho = R_n + \eta $ and because $R_n$ is non-decreasing with respect to $t$, we find
$$
v(t,x) - \eta\leq R_n(t,x) \leq R_n^\infty(x), \qquad \forall t>T, \ x>R,
$$
hence, for all $t>T$,
$$
\sup_{\delta<\vert x\vert< c t} \vert R_n(t,x) - S_{n}^{\star}\vert \leq \max\left\{\sup_{\delta < x< ct } \vert R_n^\infty(x) - S_{n}^{\star}\vert,  \sup_{\delta < x < ct}\vert v(t,x)- \sigma\vert + \vert \sigma - S_{n}^{\star} -\eta\vert \right\}.
$$
Because $v$ spreads toward $\sigma$ with speed $c_\eta>c$ and because $R_n^\infty(x)$ goes to $S_n^\star$ as $x$ goes to $+\infty$, we find that
$$
\limsup_{\delta,t\to+\infty}\sup_{\delta<\vert x\vert< c t} \vert R_n(t,x) - S_{n}^{\star}\vert \leq \vert\sigma - S_{n}^{\star} -\eta\vert.
$$
We have $\sigma \to S_{n}^{\star}$ as $\eta\to0$. Because $\eta$ is arbitrary, we find
$$
\limsup_{\delta,t\to+\infty}\sup_{\delta<\vert x\vert< c t} \vert R_n(t,x) - S_{n}^{\star}\vert = 0,
$$
and this is true for all $c\in (0,c_n)$, that is, $R_n$ spreads at least with speed $c_n$ toward $S_n^\star$.

\medskip

In addition, $R_n$ converges toward $S_n^\star$ at most with speed $c_n$. Indeed, it follows from \hyperref[LE_4_control_above_Rn_In]{Lemma \ref*{LE_4_control_above_Rn_In}} (specifically \eqref{LE_4_control_above_Rn_In_4}) that 
$$
\sup_{\verti{x} > c t}
	\verti{R_{n}(t,x)} 
	\underset{
		t \rightarrow +\infty 
		}{
		\loongrightarrow}
	0,\qquad  \forall c>c_n.
 $$
Therefore, we have proven that, if $n\in \intervalleE{0}{\Nlast-1}$ and if $\HR{n-1}$ is true, then $\HR{n}$ also holds true. \hyperref[PROP_Rn_propagation]{Proposition \ref*{PROP_Rn_propagation}} is proved by induction.  
\end{proof}

\subsection{\texorpdfstring{Proof of \hyperref[TH_know_spread_description]{Theorem \ref*{TH_know_spread_description}}}{Proof of Theorem \ref*{TH_know_spread_description}}}\label{SS_proof_of_th}
We are now in position to prove \hyperref[TH_know_spread_description]{Theorem \ref*{TH_know_spread_description}}: for $n\in \intervalleE{1}{\Nlast}$, the opinion $n$ spreads while for $n>\Nlast$, the opinion disappears. We start with a technical lemma.

\begin{lemmaNOBREAK}\label{lem_bloc}
Assume $\Nlast <+\infty$. For all $n \geq \Nlast+1$, there is $\e \in L^\infty{(\mathbb{R})}$ such that $\e(x)\underset{\vert x \vert \to +\infty}{\longrightarrow}0$ and $R_n \leq \e$.
\end{lemmaNOBREAK}

\begin{proof}
[\texorpdfstring{Proof of \hyperref[lem_bloc]{Lemma \ref*{lem_bloc}}}{Proof of Lemma \ref*{lem_bloc}}]

$\bullet$ \textit{Step 1. The case $n=\Nlast+1$.}

Let us start with proving that there is $\e \in L^\infty(\mathbb{R})$, with $\e(x)\underset{\vert x \vert \to +\infty}{\longrightarrow}0$, such that $R_{\Nlast+1}~\leq~\e$. We write $n$ instead of $\Nlast+1$ for the sake of readability. Owing to \hyperref[LE_2_Rn_sb_s_KPP_perturbed]{Lemma \ref*{LE_2_Rn_sb_s_KPP_perturbed}}, we have
	$$
	\partial_{t}R_{n} \leq
	d_{n}\Delta R_{n} +
	\mu_{n} R_{n-1}\!\Prth{1-e^{-\frac{\alpha_{n}}{\mu_{n}}R_{n}}} -
	\mu_{n}R_{n} +
	\mu_{n}I_{n}^{0},\qquad t>0,\, x\in \R.
	$$
Because $\HR{\Nlast}$ holds true, owing to \hyperref[REM_Rn_lower_than_Sn_ast_plus_eps_de_x]{Remark \ref*{REM_Rn_lower_than_Sn_ast_plus_eps_de_x}}, we have that there is $\tilde{\e} \in L^\infty(\mathbb{R})$ such that $\tilde{\e}(x)\underset{\vert x\vert \to +\infty}{\longrightarrow}0$, and
	$$
	\partial_{t}R_{n} \leq
	d_{n}\Delta R_{n} +
	\mu_{n} S_{\Nlast}^{\star}\!\Prth{1-e^{-\frac{\alpha_{n}}{\mu_{n}}R_{n}}} -
	\mu_{n}R_{n} +
	\tilde{\e},\qquad t>0, \, x\in \R.
	$$
Owing to \hyperref[lem bounded]{Lemma \ref*{lem bounded}}, we also know that $R_n \leq K_n$ for some $K_n>0$.
 
Let $u(t,x)$ be solution of 
\begin{equation}\label{eq u}
\partial_{t} u = d_{n}\Delta u + f_n(u)+\tilde{\e},\qquad t>0, \, x\in \R,
\end{equation}
with initial datum $u(0 ,\point)= K_n$. Up to increasing $K_n$ if needed, we can ensure that $K_n$ is a stationary supersolution of \eqref{eq u}. Then, is it is classical that $u(t,x) \underset{t\to+\infty}{\searrow} U(x)$, where $U$ is a stationary solution of \eqref{eq u}. We have by comparison that $R_n(t,x)\leq R_n^\infty(x) \leq U(x)$ for all $t>0, \ x\in \R$.

Let us show that $U(x)\underset{\vert x \vert \to +\infty}{\longrightarrow}0$. To this aim, we take a sequence $(x_k)_{k\in \N}$ such that $\vert x_k\vert \to +\infty$ as $k\to+\infty$, and we define the translated functions $U_k(x) = U(x+x_k)$.

Owing to classical elliptic regularity results \cite{GilbargElliptic01}, we have that, up to extraction, $U_n \underset{n\to+\infty}{\longrightarrow} U_\infty$ (this convergences is locally in $W^{2,p}(\mathbb{R})$, for all $p>1$), where $U_\infty$ solves 
$$
d_{n}\Delta U_\infty + f_{n}(U_\infty)=0,
$$
and $0\leq U_\infty\leq K_n$. The only function that satisfies this is the function everywhere equal to zero. Indeed, let $z(t)$ be solution of the ODE $\dot z = f_n(z)$ with initial datum $z(0)=K_n$. We have $z(t) \underset{t\to+\infty}{\longrightarrow}0$.

The parabolic comparison principle implies that $z(t)\geq U_\infty(x)$ for all $t>0$, and $x\in \R$. Taking the limit $t\to +\infty$ implies that $U_\infty \equiv 0$.

The result then holds true for $n=\Nlast+1$ with $\e = U$.

\medskip

$\bullet$ \textit{Step 2. The case $n>\Nlast+1$.}

We prove the result by induction: we show that, if there is $n\geq 2$ such that $R_{n-1} \leq \e$, for some $\e \in L^\infty(\mathbb{R})$, with $\e(x)\underset{\vert x \vert \to +\infty}{\longrightarrow}0$, then the same is true for $R_n$.

\medskip

Owing to \hyperref[LE_2_Rn_sb_s_KPP_perturbed]{Lemma \ref*{LE_2_Rn_sb_s_KPP_perturbed}}, we have that  
$$
	\partial_{t}R_{n} \leq
	d_{n}\Delta R_{n} +
	\mu_{n} R_{n-1}\!\Prth{1-e^{-\frac{\alpha_{n}}{\mu_{n}}R_{n}}} -
	\mu_{n}R_{n} +
	\mu_{n}I_{n}^{0}.
$$
Let $\eta>0$ be such that $-\kappa := \alpha_n \eta - \mu_{n}<0$. There is $R>0$ such that
$$
\partial_{t}R_{n} \leq
	d_{n}\Delta R_{n} +
	\mu_{n} \eta\!\Prth{1-e^{-\frac{\alpha_{n}}{\mu_{n}}R_{n}}} -
	\mu_{n}R_{n} +
	\mu_{n}I_{n}^{0},\qquad \forall t>0, \, \forall  x>R.
$$
Using the concavity of $z\mapsto \mu_{n} \eta\!\Prth{1-e^{-\frac{\alpha_{n}}{\mu_{n}}z}} -
	\mu_{n}z $, we have
\begin{equation}\label{sur Rn tech}
	\partial_{t}R_{n} \leq
	d_{n}\Delta R_{n} - \kappa R_{n} +
	\mu_{n}I_{n}^{0},\qquad \forall t>0, \, \forall x>R.
\end{equation}
Owing to \hyperref[lem bounded]{Lemma \ref*{lem bounded}}, there is $K_n>0$ such that $R_n\leq K_n$. Therefore, up to taking $A>0$ large enough and $\lambda>0$ small enough, the function
$$
v(x) = A e^{-\lambda x}
$$
is a stationary supersolution of the equation \eqref{sur Rn tech}. By comparison, up to increasing $A$ if needed, we have $R_n \leq Ae^{-\lambda x}$. Using the same arguments for $x<0$, we find that $R_n  \leq A e^{-\lambda\vert x \vert}$.

Therefore, we have proven that, if there is $\e \in L^\infty(\mathbb{R})$, with $\e(x)\underset{\vert x \vert \to +\infty}{\longrightarrow}0$, such that $R_{n-1}\leq \e$, then the same holds true for $R_n$. The lemma follows by a direct induction.
\end{proof}

\medskip

We now have all the tools to prove our main result, namely \hyperref[TH_know_spread_description]{Theorem \ref*{TH_know_spread_description}}.

\medskip

\begin{proof}[\texorpdfstring{Proof of \hyperref[TH_know_spread_description]{Theorem \ref*{TH_know_spread_description}}}{Proof of Theorem \ref*{TH_know_spread_description}}]

\medskip

$\bullet$ \textit{Proof of \eqref{EQ_th_i}: $S_{n}$ spreads at most with speed $c_n$.}
Let $n\in \intervalleE{1}{\Nlast}$. Owing to \hyperref[LE_1_control_Sn_with_Rn]{Lemma \ref*{LE_1_control_Sn_with_Rn}}, we have, $S_{n}\leq R_n$. Therefore, \hyperref[PROP_Rn_propagation]{Proposition \ref*{PROP_Rn_propagation}} directly implies that
$$
\sup_{\vert x \vert > (c_n+\e)t}\vert S_{n}(t,x)\vert\leq \sup_{\vert x \vert > (c_n+\e)t}\vert R_n(t,x)\vert \underset{t\to+\infty}{\longrightarrow} 0.
$$
This proves the point \eqref{EQ_th_i} of the theorem.

\medskip

$\bullet$ \textit{Proof of \eqref{EQ_th_ii}: $S_{n}$ converges toward $S_{n}^{\star}$ in the intermediate region $c_{n+1}t < \vert x \vert < c_n t$.}
For $n=0$, this is a consequence of \hyperref[n=1]{Remark \ref*{n=1}}. Let $n\in \intervalleE{1}{\Nlast}$. We recall that we define $c_{\Nlast+1}=0$. Owing to \hyperref[PROP_Rn_propagation]{Proposition \ref*{PROP_Rn_propagation}}, point \eqref{EQ_propo_j},  we have (when $n=\Nlast$ this is actually a consequence of \hyperref[lem_bloc]{Lemma \ref*{lem_bloc}})
$$
\sup_{\vert x \vert > (c_{n+1}+\e)t} \vert R_{n+1}\vert \underset{t\to+\infty}{\longrightarrow}0,
$$
and, owing to the point \eqref{EQ_propo_jj}, there holds
$$
\sup_{(c_{n+1}+\e)t < \vert x \vert < (c_n - \e) t} \vert R_n - S_{n}^{\star} \vert \underset{t\to+\infty}{\longrightarrow} 0.
$$
Therefore, 
$$
\sup_{(c_{n+1}+\e)t < \vert x \vert < (c_n - \e) t} \vert R_n e^{-\frac{\alpha_{n+1}}{\mu_{n+1}}R_{n+1}} - S_{n}^{\star} \vert \underset{t\to+\infty}{\longrightarrow} 0.
$$
Owing to \hyperref[LE_1_control_Sn_with_Rn]{Lemma \ref*{LE_1_control_Sn_with_Rn}}, we have that $R_n e^{-\frac{\alpha_{n+1}}{\mu_{n+1}}R_{n+1}}\leq  S_{n} \leq R_n$, from which it follows that
$$
\sup_{(c_{n+1}+\e)t < \vert x \vert < (c_n - \e) t} \vert S_{n} - S_{n}^{\star} \vert \underset{t\to+\infty}{\longrightarrow} 0.
$$
This proves the point \eqref{EQ_th_ii} of the theorem.

\medskip

$\bullet$ \textit{Proof of \eqref{EQ_th_iii}: $S_{n}$ converges toward $S_{n}^{\dagger}$ in the region $\delta <\vert x \vert < c_{n+1}t $.}
Let $n\in \intervalleE{1}{\Nlast-1}$. Let $\e>0$ be fixed. We want to prove that
$$
\sup_{\delta < \vert x \vert < (c_{n+1}-\e)t} \vert S_{n} - S_{n}^{\dagger}\vert \underset{\delta,t\to+\infty}{\longrightarrow} 0.
$$
First, owing to \eqref{eq plus tard} from the proof of \hyperref[LE_2_Rn_sb_s_KPP_perturbed]{Lemma \ref*{LE_2_Rn_sb_s_KPP_perturbed}}, we have for all $t>0$ and $x\in \R$,
\begin{equation}\label{eq zone 3}
\partial_{t} R_{n+1} = d_{n+1} \Delta R_{n+1} + \mu_{n+1}(R_{n} - S_{n}) - \mu_{n+1} R_{n+1} + \mu_{n+1}I_{n+1}^{0}.
\end{equation}
Now, let us chose three sequences $(t_k)_{k\in\N},(x_k)_{k\in\N},(\delta_k)_{k\in\N}$ such that $\delta_k, t_k \to +\infty$ and $\delta_k < \vert x_k\vert < (c_{n+1}-\e)t_n$.

\medskip

We introduce the translated functions $R_{n+1}^k = R_{n+1}(\point +t_k,\point +x_k)$, $R_{n}^k = R_{n}(\point +t_k,\point +x_k)$ and $S_{n}^k = S_{n}(\point +t_k,\point +x_k)$.

For any $t>0$ and $x\in\R$, we have, for $k$ large enough,
$$
\vert R_{n+1}(t+t_k,x+x_k) - S_{n+1}^{\star}\vert  \leq \sup_{x+x_k \leq \vert y \vert \leq (c_{n+1}-\frac{\e}{2})t_k}\vert R_{n+1}(t+t_k,y) - S_{n+1}^{\star}\vert,
$$
and, owing to \hyperref[PROP_Rn_propagation]{Proposition \ref*{PROP_Rn_propagation}}, this goes to zero as $k$ goes to $+\infty$.

Therefore, owing to parabolic regularity estimates, we have that $R_{n+1}^k$ converges (up to a subsequence) locally uniformly in $W^{2,p}$, for all $p>1$, to the function everywhere constant equal to $S_{n+1}^{\star}$. Similarly, $R_n^k$ converges toward the function everywhere constant equal to $S_{n}^{\star}$.

Therefore, because we have
$$
\partial_{t} R^k_{n+1} = d_{n+1} \Delta R^k_{n+1} +\mu_{n+1}(R^k_{n} - S^k_n) - \mu_{n+1} R^k_{n+1} + \mu_{n+1}I_{n+1}^{0}(\point+x_k),
$$
taking the limit $k\to+\infty$ in this equation yields that, up to a subsequence,
$$
\lim_{k\to+\infty} S_{n}^k(t,x)  = S_{n}^{\star} - S_{n+1}^{\star}=S_{n}^\dagger,
$$
and this convergence is locally uniform. Therefore, we have proven that, for each sequences $(x_k)_k$, $(\delta_k)_k$, $(t_k)_k$, as above, we have, up to a subsequence, $S_{n}(t_k,x_k) \underset{k\to+\infty}{\longrightarrow} S_{n}^\dagger$. The point \eqref{EQ_th_iii} of the theorem then holds true.

\medskip

$\bullet$ \textit{Proof of \eqref{EQ_th_iv}: $S_{n}$ does not spread when $n>\Nlast$.}

Now, owing to \hyperref[lem_bloc]{Lemma \ref*{lem_bloc}}, we have that, for all $n\geq \Nlast+1$, there is $\e \in L^\infty(\R)$, with $\e(x)\underset{\vert x \vert \to +\infty}{\longrightarrow}0$, such that $R_n\leq \e$. Therefore, because $S_{n} \leq R_n$, the result holds true. 
\end{proof}

\section{Qualitative properties of the propagation sequences}\label{sec_quali}

This section is dedicated to the proof of \hyperref[th_quali]{Theorem \ref*{th_quali}} concerning the qualitative properties of $\Nlast$, the maximal complexity obtained by the population.

\medskip

The key point is to study the propagation sequences given by \hyperref[DEF_sequences]{Definition \ref*{DEF_sequences}}. To do so, we can rephrase the definition of the propagation sequences as follows: let $S_0^\star$, $(d_n)_{n\in\N^\star}$, $(\alpha_n)_{n\in\N^\star}$, $(\mu_n)_{n\in\N^\star}$ be given. For $n\in \N$, define
$$
\varphi_n(x) : = {x}\bigg/\prthhhh{1-e^{-\frac{\alpha_{n+1}}{\mu_{n+1}}x}}.
$$
This function is strictly increasing for $x>0$, we have $\varphi_n(x)\underset{x\to 0}{\longrightarrow} \frac{\mu_{n+1}}{\alpha_{n+1}}$ and $\varphi_n(x)>x$ for all $x>0$.

Let $\Nlast$ and $(S_n^\star)_{n\in \llbracket0, \Nlast\rrbracket}$ be the maximal complexity and the propagation sequence given by \hyperref[DEF_sequences]{Definition \ref*{DEF_sequences}}. We have
\begin{equation}\label{suite}
S_{n}^\star= \varphi_n(S_{n+1}^\star),\qquad \forall n\in \llbracket0,\Nlast-1\rrbracket,
\end{equation}
and we can characterize $\Nlast$ as follows:
\begin{equation}\label{nlast}
\Nlast\ \textit{is the smallest integer $k\in \N$ such that}~ \frac{\alpha_{k+1}}{\mu_{k+1}}S_k^\star \leq 1,
\end{equation}
with the convention that $\Nlast = +\infty$ if for all $k\in \N$, we have $\frac{\alpha_{k+1}}{\mu_{k+1}}S_k^\star >1$.

Observe that, when all the coefficients $(\alpha_n)_{n\in\N}$ and $(\mu_n)_{n\in \N}$ are independent of $n$, we necessarily have $\Nlast <+ \infty$. Indeed, if this were not the case, then we could take the limit $n\to+\infty$ in \eqref{suite}: the function $\varphi_n$ does not depend on $n$, and the sequence $(S_n^\star)_{n\in\N}$ is non-increasing, hence converges to a limit on $\R_+^\star$, and we would reach a contradiction because $\varphi_n$ has no fixed point on $\R_+^\star$.

\medskip

We start with proving the first point of \hyperref[th_quali]{Theorem \ref*{th_quali}}, that is, we show that the maximal complexity is a non-decreasing function of the initial population $S_0^\star$ and of the transmission parameters $(\alpha_n)_{n\in \N^\star}$ and non-increasing with respect to the recovery parameters $(\mu_n)_{n\in \N^\star}$

\medskip

\begin{proof}[\texorpdfstring{Proof of \hyperref[th_quali]{Theorem \ref*{th_quali} \raisebox{0.4mm}{$\displaystyle \scalebox{0.8}{$\displaystyle \boxed{1} $} $}\hspace{0.4mm}}}{Proof of Theorem \ref*{th_quali} 1.}]
For $n\in\N^\star$, denote
\begin{equation}\label{EQ_PHI_N}
\varphi_n(x) : = {x}\bigg/\prthhhh{1-e^{-\frac{\alpha_{n+1}}{\mu_{n+1}}x}}
\qquad
\text{and}
\qquad
\bara{\varphi_{n}}(x) : = {x}\bigg/\prthhhh{1-e^{-\frac{\bara{\alpha_{n+1}}}{\bara{\mu_{n+1}}}x}}.
\end{equation}
Owing to the hypotheses, we have $\varphi_{n} \leq \bara{\varphi_{n}}$.

\medskip

We denote $\Nlast := \Nlast(S_0^\star,\alpha_n,\mu_n)$ and  $\bara{\Nlast} := \Nlast(\bara{S_0^\star},\bara{\alpha_n},\bara{\mu_n})$ the maximal complexities reached by the systems with each set of parameters, and we denote $(S_n^\star)_{n\in \llbracket1,\Nlast\rrbracket}$ and $(\bara{S_n^\star})_{n\in \llbracket1,\bara{\Nlast}\rrbracket}$ the corresponding propagation sequences.

\medskip

Since $S_0^\star\leq \bara{S_0^\star}$, and because the functions $\varphi_{n}$ and $\bara{\varphi_{n}}$ are non-decreasing and $\varphi_{n} \leq \bara{\varphi_{n}}$, it follows from \eqref{suite} that, for each $n\in \llbracket1, \min\{\Nlast,\bara{\Nlast}\}\rrbracket$, we have
$$
S_n^\star \leq \bara{S_n^\star},
$$
and therefore
$$\prth{\alpha_{n+1}/\mu_{n+1}}S_n^\star \leq \prth{\bara{\alpha_{n+1}}/\bara{\mu_{n+1}}}\bara{S_n^\star}.$$
Owing to \eqref{nlast}, this gives
$$
\Nlast\leq \bara{\Nlast}.
$$
\end{proof}

We now turn to the second point of \hyperref[th_quali]{Theorem \ref*{th_quali}}, that is we show that, if the parameters are chosen adequately, it is possible to have $\Nlast =+\infty$. Moreover, it is possible to ensure that any proportion of the initial population can reach infinite complexity.

\medskip

\begin{proof}[\texorpdfstring{Proof of \hyperref[th_quali]{Theorem \ref*{th_quali} \raisebox{0.4mm}{$\displaystyle \scalebox{0.8}{$\displaystyle \boxed{2} $} $}\hspace{0.4mm}}}{Proof of Theorem \ref*{th_quali} 2.}]
Let us start by observing from \eqref{EQ_PHI_N}, that, for all $n\in \N^\star$,
$$
\varphi_{n}(x) \leq x + \frac{1}{\lambda_{n+1}},
$$
Where $\lambda_n : = \frac{\alpha_n}{\mu_n}$.
Therefore, owing to \eqref{suite}, we find, for all $n\in \N^\star$,
$$
S_n^\star \geq S_0^\star - \sum_{k=1}^n \frac{1}{\lambda_k}.
$$
Hence, up to choosing a sequence $(\lambda_n)_{n\in \N^\star}$ so that $\sum_{k=1}^{+\infty} \frac{1}{\lambda_k} \leq \e$, we ensure that 
$$
\lim_{n\to +\infty} S_n^\star \geq S_0^\star - \e.
$$
To conclude the proof, we need to apply \hyperref[TH_know_spread_description]{Theorem \ref*{TH_know_spread_description}}, and to do so, we need the sequence of the speeds $(c_n)_{n \in \N^\star}$ to be decreasing (\hyperref[H1]{Assumption \ref*{H1}}). One way to do this is to choose the diffusions $(d_n)_{n\in\N}$ so that this is true. We could also take all the diffusions equal (which is more natural from the modeling point of view) and multiply each $\alpha_n$ and $\mu_n$ by a coefficient $\e_n>0$, so that the ratios $\lambda_n$ are not changed while the speeds $c_n = 2\sqrt{d(\alpha_{n+1}\e_{n+1} S_n^\star - \e_{n+1}\mu_{n+1})}$ are decreasing.
\end{proof}

We now turn to the third point of \hyperref[th_quali]{Theorem \ref*{th_quali}}, that is, we show that the maximal complexity $\Nlast$ is equivalent to
${e^{{\alpha S_0^\star}/{\mu}}}/\prth{{\alpha S_0^\star}/{\mu}} = {e^{\mathcal{R}_0}}/{\mathcal{R}_0}$.

\medskip

\begin{proof}[\texorpdfstring{Proof of \hyperref[th_quali]{Theorem \ref*{th_quali} \raisebox{0.4mm}{$\displaystyle \scalebox{0.8}{$\displaystyle \boxed{3} $} $}\hspace{0.4mm}}}{Proof of Theorem \ref*{th_quali} 3.}]
Let $S_{0}^{\star}>0$ fixed. Let $(S_{n}^{\star})_{n\in\N^\star}$ be the propagation sequence, as given in \hyperref[DEF_sequences]{Definition \ref*{DEF_sequences}}. 

First, observe that, because the sequence $(S_{n}^{\star})_{n\in\N^\star}$ is decreasing, then the sequence of the speeds $(c_n)_{n\in\N^\star}=(2\sqrt{d(\alpha S_n^\star - \mu)})_{n\in\N^\star}$ is also strictly decreasing, hence \hyperref[H1]{Assumption \ref*{H1}} is verified. Therefore, \hyperref[TH_know_spread_description]{Theorem \ref*{TH_know_spread_description}} applies. We let $\lambda = \frac{\alpha}{\mu}$.

\medskip

By denoting $\varphi(x) : = {x}/\prth{1 - e^{-\lambda x}}$, we have by definition
$$
S_{n-1}^{\star} = \varphi(S_{n}^{\star}),
\qquad
\forall n\in \llbracket1,\Nlast\rrbracket.
$$
Because the $\alpha_n$ and $\mu_{n}$ are independent of $n$, as explained in the beginning of this section, we have $\Nlast<+\infty$. Owing to \eqref{nlast}, we have that
$$
S_{\Nlast}^{\star}\leq \frac{1}{\lambda}
\qquad
\text{and}
\qquad
S_{\Nlast-1}>\frac{1}{\lambda}.
$$
Therefore, because $\varphi$ is increasing,
\begin{equation}\label{S suite 1}
\frac{1}{\lambda}\, \leq \,S_{\Nlast-1}^{\star}\, \leq \,\varphi(\frac{1}{\lambda}).
\end{equation}
We denote $(u_n)_{n\in \N}$ the sequence defined by induction $u_{n+1} = \varphi(u_n)$ with $u_0 = \frac{1}{\lambda}$. Because $\varphi(x)>x$ the sequence $(u_n)_{n\in \N}$ is increasing. It follows that it diverges to $+\infty$.

Applying $\varphi$ in in \eqref{S suite 1} $\Nlast-1$ times, we get
\begin{equation}\label{equiv suite}
u_{\Nlast-1}
\, \leq \,
S^{\star}_0
\, \leq \,
u_{\Nlast}.
\end{equation}
Now, define $\Phi(x) = \int_0^x \frac{e^{\lambda z} - 1}{z}dz$. We have
\begin{equation}\label{ineg suite}
\vert \Phi(u_{n+1}) - \Phi(u_n) - \Phi^\prime(u_n)(u_{n+1} - u_n) \vert
\, \leq \,
\frac{1}{2}\sup_{z\in [u_n,u_{n+1}]} \vert \Phi^{\prime\prime}(z)\vert (u_{n+1} - u_n)^2.
\end{equation}
We have $u_{n+1} - u_n =  \frac{u_n}{e^{\lambda u_n} - 1} = \frac{1}{\Phi^\prime(u_n)}$, and this goes to zero when $n$ goes to $+\infty$. 

On the other hand, it is clear that $\Phi^{\prime\prime}$ is increasing for $z>0$. Therefore,
$$
\frac{1}{2}\sup_{z\in [u_n,u_{n+1}]}\vert \Phi^{\prime\prime}(z)\vert (u_{n+1} - u_n)^2
\, \leq \,
\frac{\Phi^{\prime\prime}(u_{n+1})}{2\Phi^{\prime}(u_n)^2}
\, \leq \,
\lambda \frac{u_n^2}{u_{n+1}}\frac{e^{\lambda u_{n+1}}}{(e^{\lambda u_n}-1)^2},
$$
and because $u_{n+1}-u_n$ goes to zero when $n$ goes to $+\infty$, we have that the right-hand side in \eqref{ineg suite} goes to zero.

Therefore, $\Phi(u_{n+1}) - \Phi(u_n) \to 1$ when $n$ goes to $+\infty$, and the Cesàro lemma implies
$$
\frac{\Phi(u_n)}{n} \underset{n\to+\infty}{\longrightarrow} 1.
$$
Now, applying $\Phi$ (which in increasing) to \eqref{equiv suite}, we get
$$
\Nlast(S_{0}^{\star})
\;
\underset{{S_{0}^{\star} \to +\infty}}{\scalebox{1.8}{$\displaystyle \sim $}}
\;
\Phi(S_{0}^{\star}).
$$
By observing that
$
\Phi(x)
\;
\underset{{x \to +\infty}}{\scalebox{1.8}{$\displaystyle \sim $}}
\;
\frac{e^{\lambda x}}{\lambda x}
$,
the result follows.
\end{proof}

\medskip

\noindent\textbf{Acknowledgements.}
\href{https://samueltreton.fr/english/}{Samuel Tréton} would like to express his gratitude to \href{https://vcalvez.perso.math.cnrs.fr/wacondy.html}{Vincent Calvez} for supporting his postdoctoral position.
This project has received funding from the European Research Council (ERC) under the European Union’s Horizon 2020 research and innovation programme (grant agreement No 865711).
This study contributes to the IdEx Université de Paris ANR-18-IDEX-0001. The research leading to these results has received funding from the ANR project “ReaCh” (ANR-23-CE40-0023-01).

\medskip

\nocite{*} 

\bibliographystyle{siam}  
\bibliography{biblio}

\end{document}